\documentclass[]{article}

\usepackage{amsmath, amsthm, amssymb, latexsym}
\usepackage{todonotes} 
\usepackage{graphicx}
\usepackage{overpic}
\usepackage{enumitem}
\usepackage[hidelinks]{hyperref} 
\usepackage[capitalise, nameinlink, noabbrev]{cleveref}
\usepackage[margin=1in]{geometry}
\usepackage{array}
\usepackage{comment}

\title{Matrix-valued orthogonal polynomials related to hexagon tilings}

\author{Alan Groot
\footnote{Department of Mathematics, Katholieke Universiteit Leuven, Belgium, Email: alan.groot@kuleuven.be. 
	Supported by long term structural funding-Methusalem grant of the Flemish
	Government.}	 
\and  Arno B.J. Kuijlaars
\footnote{Department of Mathematics, Katholieke Universiteit Leuven, Belgium, Email: arno.kuijlaars@kuleuven.be. Supported by long term structural funding-Methusalem grant of the Flemish
	Government, and by FWO Flanders projects
	EOS 30889451, G.0864.16 and  G.0910.20.}	  
}

\newcommand{\C}{\mathbb{C}}
\newcommand{\R}{\mathbb{R}}
\newcommand{\Z}{\mathbb{Z}}

\newcommand{\ds}{\displaystyle}

\newcommand{\bigO}{O}

\renewcommand{\d}{\text{d}}
\renewcommand{\Re}{\operatorname{Re}}
\renewcommand{\Im}{\operatorname{Im}}

\newcommand{\RHP}{{\operatorname{RHP}}}
\newcommand{\LHP}{{\operatorname{LHP}}}

\newcommand{\Bal}{\operatorname{Bal}}

\newcommand{\supp}{\operatorname{supp}}

\newcommand{\zplus}{z_+}
\newcommand{\zminus}{z_-}

\newtheorem{theorem}{Theorem}
\newtheorem{proposition}[theorem]{Proposition}
\newtheorem{corollary}[theorem]{Corollary}
\newtheorem{lemma}[theorem]{Lemma}

\theoremstyle{definition}
\newtheorem{definition}[theorem]{Definition}
\newtheorem{rhp}{Riemann-Hilbert problem}

\theoremstyle{remark}

\numberwithin{theorem}{section}
\numberwithin{equation}{section}

\newcommand{\prob}{\mathbb{P}}

\begin{document}
	
	\maketitle

\begin{abstract}
In this paper, we study a class of matrix-valued orthogonal polynomials (MVOPs) 
that are related to $2$-periodic lozenge tilings of a hexagon.
The general model depends on many parameters. 
In the cases of constant and $2$-periodic parameter values we show that the 
MVOP can be expressed in terms of scalar polynomials
with non-Hermitian orthogonality on a closed contour in the 
complex plane. 

The $2$-periodic hexagon tiling model with a constant parameter 
has a phase transition in the large size limit. This is reflected
in the asymptotic behavior of the MVOP as the degree tends to infinity.
The connection with the scalar orthogonal polynomials allows us
to find the limiting behavior of the zeros of the determinant of 
the MVOP. The zeros tend to a curve $\widetilde{\Sigma}_0$ in
the complex plane that has a self-intersection.

The zeros of the individual entries of the MVOP show a different
behavior and we find the limiting zero distribution of
the upper right entry under a geometric condition on the curve
$\widetilde{\Sigma}_0$ that we were unable to prove, but that
is convincingly supported by numerical evidence. 
\end{abstract}


\section{Introduction and reduction to scalar orthogonality} \label{sec:algebraic}

The aim of this paper is to study a class of matrix-valued orthogonal polynomials 
(or MVOPs for short) that are related to weighted lozenge 
tilings of a hexagon.  We explain the connection in more detail
in \cref{sec:tilings}. 
In this section we introduce the MVOPs and state our results
on the reduction to scalar orthogonality that are valid for finite 
degrees of the MVOPs. 
We state asymptotic results in \cref{sec:asymptotic}. 

\subsection{Existence and uniqueness of MVOP}
The matrix-valued orthogonality coming from the hexagon tiling
models takes the particular form
\begin{equation} \label{eq:MVOP1}
\frac{1}{2\pi i} \oint_{\gamma} P_k(z) W(z) z^j dz = 0_d,
	\qquad j=0,1, \ldots, k-1, \end{equation}
where $P_k(z) = z^k I_d + \cdots$ is a monic matrix-valued polynomial of degree $k$
and size $d \times d$, $W$ is a $d \times d$ weight matrix whose entries
are rational functions, and $\gamma$ is a closed contour in the complex plane 
that encircles the poles of $W$.
The integrand in \eqref{eq:MVOP1} is matrix-valued and
the integral is to be taken entrywise. 
We use $0_d$ to denote the zero matrix of size $d \times d$,
and $I_d$ for the identity matrix. The type of matrix-valued orthogonality
originates from the work \cite{duits2017periodic} where it was applied
to $2$-periodic tilings of the Aztec diamond. 

The development of the theory of MVOPs dates back to the 1940s, see the survey
 \cite{damanik} and the many references therein.
The more recent research on MVOPs (mostly from around the mid-1990s) deals 
with orthogonality on the real line with a non-negative definite weight matrix $W$, 
see  for example \cite{aptekarev_nikishin, duran_markovs_1996, duran_lopez-rodriguez-MOPs, duran_orthogonal_1995, sinap_orthogonal_1996}.    
In that case, the
existence and uniqueness of the MVOP are an easy consequence of a
(matrix-valued) Gram-Schmidt orthogonalization process.
In contrast, the orthogonality \eqref{eq:MVOP1} is not
associated with a matrix-valued positive definite scalar product and
existence and uniqueness of MVOPs is not a priori guaranteed. Indeed,
if $W$ is rational (as is the case in this paper), 
then any $P_k$ that cancels
the poles of $W$ (in the sense that $P_k W$ is entire) will
satisfy \eqref{eq:MVOP1} by Cauchy's theorem. Hence uniqueness
of MVOPs is certainly lost for $k$ large enough.
Observe also that the integrand in \eqref{eq:MVOP1} is analytic, except for poles of $W$, 
and therefore the contour of integration can be deformed as long 
as we do not cross any poles of $W$.

In the paper we focus on a special class of examples of size $2 \times 2$. We let
\begin{equation} \label{eq:weightWalpha} 
	W_{\alpha}(z) = \begin{pmatrix} \alpha & 1 \\ z & 1 \end{pmatrix},
	\qquad \alpha \in \mathbb R, 
\end{equation}
and
\begin{equation} \label{eq:weightW} 
	W(z) = W_{\vec{\alpha}, K,L} =  \frac{W_{\alpha_1}(z) \cdots W_{\alpha_L}(z)}{z^{K}},
\end{equation}
with non-negative integers $K$ and $L$, and $\vec{\alpha} = 
(\alpha_1, \ldots, \alpha_L) \in \mathbb R^L$. The contour $\gamma$
encircles the origin once in the positive direction.

\begin{proposition} \label{prop:exist}
	Let $K,L$ be non-negative integers and let $\alpha_j > 0$ 
	for every $j=1, \ldots, L$. Let $k$
	be a non-negative integer satisfying 
	\begin{equation} \label{eq:kcondition}  
	2K-L \leq 2k \leq 2K.  \end{equation}
	Let $\gamma$ be a closed contour in the complex plane going
	around $0$ once in the positive direction.
	 Then the monic matrix-valued polynomial $P_k$ of degree $k$, satisfying
	\eqref{eq:MVOP1} with weight matrix \eqref{eq:weightW} exists
	and is unique.
\end{proposition}
The proof is based on the connection with a weighted  lozenge
tiling model of an $ABC$-hexagon with side lengths
$A = 2k$, $B= L-2K+2k$, and $C = 2K-2k$. The inequalities 
\eqref{eq:kcondition} imply that the side lengths are non-negative. 
We explain this connection in \cref{sec:tilings}.
Once we know that the tiling model exists, we can apply
\cite[Lemma 4.8]{duits2017periodic} to obtain existence
and uniqueness of $P_k$, see \cref{subsec:tilings4}.

\subsection{Scalar orthogonality}

In several cases we can express the MVOP $P_k$ in terms of 
meromorphic functions with an orthogonality on a Riemann surface. 
In case the Riemann surface has genus $0$ we can map it to the Riemann 
sphere, and we obtain orthogonal rational functions on $\mathbb C$, 
that may  lead to scalar orthogonal polynomials, depending on the situation.
 
The reduction to scalar orthogonality on a Riemann surface is not new.
It is implicit in \cite{duits2017periodic}, and made 
explicit by Charlier \cite{charlier2020matrix} where the
emphasis is on the Christoffel-Darboux kernel associated with
matrix-valued orthogonality that
originates with \cite{delvaux_average}.
Here we focus on the polynomials
in two particular cases associated with the weights \eqref{eq:weightW}
that give rise to genus $0$ Riemann surfaces and orthogonality
in the complex plane.   
We also give direct proofs, while \cite{charlier2020matrix} relies
on the Riemann-Hilbert problem for MVOPs, that was first considered in \cite{cascon_manas_2012,  grunbaum_iglesia_MF_2011}, see also \cite{delvaux_average}.
The first case is where $\alpha_j=\alpha >0$ for all $j$. 
\begin{theorem} \label{thm:theorem1}
	Let $K,L, k$, be non-negative integers
	such that \eqref{eq:kcondition}  is satisfied.
	Suppose $\alpha_j = \alpha > 0$ for every $j=1,\ldots, L$.
	Let $\gamma$ be a closed contour going
	around $0$ once in the positive direction.
	Let $P_k$ be the MVOP of degree $k$ with weight
	function
	\begin{equation} \label{eq:matrixW1} \frac{W_{\alpha}(z)^L}{z^K} 
	\end{equation}
	on $\gamma$ that uniquely
	exists by Proposition \ref{prop:exist}.
	Let $\beta = \frac{1-\alpha}{2}$. 
	
	\begin{enumerate}
		\item[\rm (a)] 
	Then
	\begin{equation} \label{eq:Pkformula1} 
		P_k(\zeta^2 - \beta^2) \begin{pmatrix}
		1 & 1 \\
		\zeta + \beta & -\zeta + \beta 	\end{pmatrix}
		= \begin{pmatrix} q_{2k}(\zeta) & q_{2k}(-\zeta)
		\\ q_{2k+1}(\zeta) & q_{2k+1}(-\zeta) 
		\end{pmatrix} 
			 \end{equation}
	where $q_{2k}$ and $q_{2k+1}$ are monic polynomials
	of degrees $2k$ and $2k+1$ respectively, satisfying 
	\begin{equation} \label{eq:scalarOP1} 
		\frac{1}{2\pi i} 
		\oint_{\gamma_{R}} q_{2k+\epsilon}(\zeta)
		\zeta^j 	\frac{(\zeta + \alpha + \beta)^L}{(\zeta^2-\beta^2)^K} d \zeta = 0, \qquad
		j=0,1, \ldots, 2k-1, \quad
		\epsilon = 0, 1, \end{equation}
	where $\gamma_R$ is the circle of radius $R > |\beta|$
	around $0$.
	\item[\rm (b)]	
	Conversely, if $q_{2k}$ and $q_{2k+1}$ are monic
	polynomials of the indicated degrees satisfying \eqref{eq:scalarOP1}, then there is a constant 
	\begin{equation} \label{eq:constantc} 
		c = \beta - \frac{1}{2\pi i}
		\oint_{\gamma} \frac{q_{2k+1}(\zeta)}{\zeta^{2k+1}} d\zeta
		\end{equation}
	such that,  with $\zeta = (z + \beta^2)^{1/2}$
	the principal branch of the square root (i.e., $\Re \zeta > 0$), we have
	\begin{align} \nonumber 
		P_k(z) & = \begin{pmatrix} 1 & 0 \\ c & 1 \end{pmatrix}
			\begin{pmatrix} q_{2k}(\zeta) & q_{2k}(-\zeta) \\
			q_{2k+1}(\zeta) & q_{2k+1}(-\zeta)
			\end{pmatrix} 
				\begin{pmatrix} 1 & 1 \\ \zeta + \beta & 
				-\zeta + \beta \end{pmatrix}^{-1} \\
			& =  \begin{pmatrix} 1 & 0 \\ c & 1 \end{pmatrix}
			\begin{pmatrix} 
			\frac{q_{2k}(\zeta) + q_{2k}(-\zeta)}{2}
				& \frac{q_{2k}(\zeta) - q_{2k}(-\zeta)}{2 \zeta}
			\\
			\frac{q_{2k+1}(\zeta) + q_{2k+1}(-\zeta)}{2}
			& \frac{q_{2k+1}(\zeta) - q_{2k+1}(-\zeta)}{2 \zeta}
			\end{pmatrix}
			\begin{pmatrix} 1 & 0 \\ -\beta & 1
			\end{pmatrix}.	\label{eq:scalarOP2}
	\end{align}  
	\end{enumerate} 
\end{theorem}

Note that the polynomials appearing in the middle matrix of
the right-hand side of \eqref{eq:scalarOP2} are even  in $\zeta$
and therefore they are indeed polynomial in the 
variable $z = \zeta^2 - \beta^2$. The diagonal
entries are monic of degree $k$ in $z$, the $(1,2)$-entry has degree $\leq k-1$ in $z$, and the $(2,1)$-entry is
a polynomial of degree $\leq k$ whose leading
coefficient is equal to the coefficient of $\zeta^{2k}$
in the polynomial $q_{2k+1}$. The choice of $c$
in \eqref{eq:constantc} then guarantees that 
the $(2,1)$-entry of the product \eqref{eq:scalarOP2}
has degree $\leq k-1$. As a result the right-hand side 
of \eqref{eq:scalarOP2} is a monic matrix valued polynomial of 
degree $k$ in $z$.

The identities \eqref{eq:scalarOP1} are scalar orthogonality
properties of the polynomials $q_{2k}$ and $q_{2k+1}$.  
They are both orthogonal to polynomials of degree $\leq 2k-1$
with respect to the rational weight
\begin{equation} \label{eq:scalarwKL} 
w_{\alpha, K,L}(\zeta) =  \frac{(\zeta+ \alpha + \beta)^L}{(\zeta^2-\beta^2)^K},
\qquad \beta = \frac{1-\alpha}{2}, 
\end{equation}
on $\gamma_R$. More precisely, \cref{thm:theorem1} 
has the following immediate corollary. 

\begin{corollary} 
	\begin{enumerate}
		\item[\rm (a)] 
	The monic polynomial $q_{2k}$ of degree $2k$
	satisfying \eqref{eq:scalarOP1} with $\epsilon =0$ 
	exists uniquely, and it is the unique scalar 
	orthogonal polynomial of degree $2k$ with respect
	to the weight \eqref{eq:scalarwKL} on $\gamma_R$.
	\item[\rm (b)] 
	The monic polynomial $q_{2k+1}$ of degree $2k+1$
	satisfying \eqref{eq:scalarOP1} with $\epsilon = 1$
	exists, but it is not unique.
	If $q_{2k+1}$ satisfies \eqref{eq:scalarOP1} then
	so does $q_{2k+1} + c q_{2k}$ for any $c$, and this
	is the only freedom we have.
	\item[\rm (c)] If
	\[ \frac{1}{2\pi i} \oint_{\gamma_R}
		q_{2k}(\zeta) \zeta^{2k}
			\frac{(\zeta + \alpha +\beta)^L}{(\zeta^2-\beta^2)^K} d \zeta \neq 0, \]
	then the constant $c$ can be chosen such that
	the polynomial of degree $2k+1$ satisfies
	\eqref{eq:scalarOP1} with $j=2k$ as well,
	and then the  degree $2k+1$ scalar orthogonal 
	polynomial with weight \eqref{eq:scalarwKL} on $\gamma_R$ uniquely exists.
\end{enumerate}
\end{corollary} 

The scalar weight \eqref{eq:scalarwKL} is rational with a
zero at $-\alpha - \beta = -  \frac{1+\alpha}{2}$ of order $L$
and two poles at $\pm \beta = \pm \frac{1-\alpha}{2}$ of order $K$.
In case $\alpha = 1$, the two poles coincide and then the scalar orthogonal
polynomials can be expressed in terms of Jacobi polynomials. 
Another special (limiting) case is $\alpha = 0$, since then one
of the poles coincides with the zero and the weight reduces to a weight
with one zero of order $L-K$ (if $L > K$) and one pole of order $K$. Again
the scalar orthogonal polynomials can be expressed in terms of
Jacobi polynomials. 

A reduction to classical orthogonal polynomials may also appear
for MVOP on the real line as in
\cite[Theorem 5.1]{duran_grunbaum_2005}, where certain MVOPs
are expressed in terms of Hermite polynomials.

The proof of Theorem \ref{thm:theorem1} essentially relies on the spectral
decomposition of the matrix $W_{\alpha}$  from \eqref{eq:weightWalpha}.
We have
\begin{align} \label{eq:Walpha}
	W_{\alpha} & = E \Lambda E^{-1}, 
\end{align}
with
\begin{align}
	\label{eq:Lambda}
	\Lambda & = \begin{pmatrix} \lambda_1 & 0 \\ 0 & \lambda_2
	\end{pmatrix}, \qquad  
	\lambda_{1,2}(z) = \pm (z + \beta^2)^{1/2}  + \alpha + \beta, 	\\ \label{eq:Ealpha}
	E & = \begin{pmatrix} 1 & 1 \\ 
		\lambda_1-\alpha & \lambda_2 - \alpha \end{pmatrix}.
\end{align}
The eigenvalues \eqref{eq:Lambda} are the two branches of a meromorphic
function $\zeta +\alpha + \beta$ defined on the two-sheeted Riemann surface 
associated with the equation 
\[ \zeta^2 = z+\beta^2. \] 
Also the eigenvectors of $W_{\alpha}$, that are in the columns
of \eqref{eq:Ealpha}, are meromorphic on the Riemann surface
that has genus zero and thus can be mapped conformally to the
Riemann sphere.

\subsection{Periodic parameters $\alpha_1$, $\alpha_2$}

We have a similar result for the second case
where the parameters $\alpha_j$ alternate between $\alpha_1$ and $\alpha_2$.
In this case, there is a connection between the MVOP and $2 \times 2$-periodic tilings of the hexagon, see \cite{charlier2020doubly}.
Here we rely on the fact that the eigenvalues
of $W_{\alpha_1} W_{\alpha_2}$ are again 
meromorphic on a Riemann surface of genus $0$.
An extension to sequences $(\alpha_j)_j$ with higher periodicity
fails since then the eigenvalues live on a higher genus
Riemann surface. 

\begin{theorem} \label{thm:theorem2}
	Let $K,L, k$ be non-negative integers
	such that \eqref{eq:kcondition}  is satisfied.
	Suppose $L$ is even, and $\alpha_{2j-1} = \alpha_1 > 0$,
	$\alpha_{2j} = \alpha_2 > 0$ for every $j=1,\ldots, L/2$.
	Let $\gamma$ be a closed contour going around $0$ once in the positive direction.
	Let $P_k$ be the MVOP of degree $k$ with weight
	function
	\begin{equation} \label{eq:matrixW2} \frac{(W_{\alpha_1}(z) W_{\alpha_2}(z))^{L/2}}{z^K} 
	\end{equation}
	on $\gamma$ that uniquely exists by \cref{prop:exist}.
	Let 
	\begin{equation} \label{eq:beta2}
	\beta = \frac{1-\alpha_1 \alpha_2}{2 \sqrt{1+\alpha_1} 
		\sqrt{1+\alpha_2}}. \end{equation}
	\begin{enumerate}
		\item[\rm (a)] 
		Then
		\begin{equation} \label{eq:Pkformula2} 
		\begin{pmatrix} \sqrt{\frac{1+\alpha_2}{1+\alpha_1}} & 0 \\
		0 & 1 \end{pmatrix} 
		P_k(\zeta^2 - \beta^2) \begin{pmatrix}
		\sqrt{\frac{1+\alpha_1}{1+\alpha_2}} & 
		\sqrt{\frac{1+\alpha_1}{1+\alpha_2}} \\
		\zeta + \beta & -\zeta + \beta 	\end{pmatrix}
		= \begin{pmatrix} q_{2k}(\zeta) & q_{2k}(-\zeta)
		\\ q_{2k+1}(\zeta) & q_{2k+1}(-\zeta) 
		\end{pmatrix} 
		\end{equation}
		where $q_{2k}$ and $q_{2k+1}$ are monic polynomials
		of degrees $2k$ and $2k+1$ respectively, with the scalar orthogonality
		\begin{equation} \label{eq:scalarOP21} 
		\frac{1}{2\pi i} 
		\oint_{C_{R}} q_{2k+\epsilon}(\zeta)
		\zeta^j 	
		\frac{\left((\zeta - \eta_1)(\zeta-\eta_2)  \right)^{L/2}}{(\zeta^2-\beta^2)^K} d \zeta = 0, \qquad
		j=0,1, \ldots, 2k-1, \quad
		\epsilon = 0, 1, \end{equation}
		where
		\[ \eta_{j} =
		 	- \frac{1 + 2 \alpha_j + \alpha_1 \alpha_2}{2\sqrt{1+\alpha_1}
		 		\sqrt{1+\alpha_2}}, \qquad j=1,2,	\]
		and $\gamma_R$ is the circle of radius $R > |\beta|$
		around $0$.
		\item[\rm (b)]	
		Conversely, if $q_{2k}$ and $q_{2k+1}$ are monic
		polynomials of the indicated degrees, satisfying \eqref{eq:scalarOP1} then there is a constant 
		\begin{equation} \label{eq:constantc} 
		c = \beta - \frac{1}{2\pi i}
		\oint_{\gamma} \frac{q_{2k+1}(\zeta)}{\zeta^{2k+1}} d\zeta
		\end{equation}
		such that,  with $\zeta = (z + \beta^2)^{1/2}$
		the principal branch of the square root (i.e., $\Re \zeta > 0$), we have
		\begin{align} \nonumber 
		P_k(z) & = 
		\begin{pmatrix} \sqrt{\frac{1+\alpha_1}{1+\alpha_2}} & 0
		\\ 0 & 1 \end{pmatrix} \begin{pmatrix} 1 & 0 \\ c & 1 \end{pmatrix}
		\begin{pmatrix} q_{2k}(\zeta) & q_{2k}(-\zeta) \\
		q_{2k+1}(\zeta) & q_{2k+1}(-\zeta)
		\end{pmatrix} 
		\begin{pmatrix} \sqrt{\frac{1+\alpha_1}{1+\alpha_2}} & 
		\sqrt{\frac{1+\alpha_1}{1+\alpha_2}} \\ \zeta + \beta & 
		-\zeta + \beta \end{pmatrix}^{-1} \\
		& = 
		\begin{pmatrix} \sqrt{\frac{1+\alpha_1}{1+\alpha_2}} & 0
		\\ 0 & 1 \end{pmatrix} \begin{pmatrix} 1 & 0 \\ c & 1 \end{pmatrix}
		\begin{pmatrix} 
		\frac{q_{2k}(\zeta) + q_{2k}(-\zeta)}{2}
		& \frac{q_{2k}(\zeta) - q_{2k}(-\zeta)}{2 \zeta}
		\\
		\frac{q_{2k+1}(\zeta) + q_{2k+1}(-\zeta)}{2}
		& \frac{q_{2k+1}(\zeta) - q_{2k+1}(-\zeta)}{2 \zeta}
		\end{pmatrix}
		\begin{pmatrix} 1 & 0 \\ -\beta & 1
		\end{pmatrix}
		\begin{pmatrix} \sqrt{\frac{1+\alpha_2}{1+\alpha_1}} & 0 \\ 0 & 1 \end{pmatrix}.	\label{eq:scalarOP22}
		\end{align}  
	\end{enumerate} 
\end{theorem}
It is an easy check to see that \cref{thm:theorem2}
reduces to  \cref{thm:theorem1} in case $\alpha_1 = \alpha_2$.

The proof of  \cref{thm:theorem2} is essentially the same as that
of \cref{thm:theorem1}. The crucial property is that $W_{\alpha_1}(z)W_{\alpha_2}(z)$ has the
eigenvalues
\[ \lambda_{1,2}	
		= \pm \sqrt{(1+\alpha_1)(1+\alpha_2)}
			\left(z+ \beta^2\right)^{1/2} + z + \frac{1+\alpha_1 \alpha_2}{2} \]
with $\beta$ given by \eqref{eq:beta2},	Thus the eigenvalues
and eigenvectors  live on the Riemann surface associated
with the equation $\zeta^2 = z+ \beta^2$, which again
has genus $0$. We do not give more details on the proof of
\cref{thm:theorem2}.

In \cite{charlier2020doubly}, Charlier studies a $2 \times 2$-periodically weighted tiling model of the hexagon leading to matrix-valued
orthogonality with weight matrix $W^{2N}(z)/z^{2N}$ where 
\[ W(z) = \begin{pmatrix} 1 & 1 \\ 
		\alpha z & 1 \end{pmatrix} \begin{pmatrix} \alpha^2 & \alpha \\
		z & 1 \end{pmatrix}, \qquad \alpha > 0, \]
which is similar to \eqref{eq:matrixW2}. 
The matrix-valued orthogonality can be reduced to scalar orthogonality
in this case as well, with a scalar weight
\[ \left( \frac{(\zeta- \alpha c) (\zeta-\alpha c^{-1})}
		{\zeta(\zeta-c)(\zeta-c^{-1})} \right)^{2N},
			\qquad c = \sqrt{\frac{\alpha}{1-\alpha + \alpha^2}} \]
on a contour going around $c$ and $c^{-1}$ but not around $0$. 
See \cite[Section 3]{charlier2020doubly} where this is discussed on the 
level of the reproducing kernels. 

Higher periodicity in the parameters will lead to
a higher genus Riemann surface. For example in the case
of $3$-periodicity with $\alpha_{3j+1}= \alpha_1$, $\alpha_{3j+2}=\alpha_2$,
$\alpha_{3j+3} = \alpha_3$ we find that the eigenvalues
of $W_{\alpha_1}(z) W_{\alpha_2}(z) W_{\alpha_3}(z)$ 
live on the genus $1$ Riemann surface associated with
\[ \zeta^2 = 4z^3 + a z^2 + b z + c \]
where $a,b,c$ are explicit in terms of the parameters $\alpha_j$,
$j=1,2,3$, 
\begin{align*}
	a & = (1+\alpha_1 + \alpha_2 + \alpha_3)^2 + 8, \\
	b & = 6 \alpha_1 \alpha_2 \alpha_3 + \alpha_1 \alpha_2 \alpha_3(\alpha_1 + \alpha_2 +\alpha_3)
		+ 4(\alpha_1 \alpha_2 + \alpha_1 \alpha_3+\alpha_2\alpha_3)
			+ 2(\alpha_1+\alpha_2 + \alpha_3) + 6, \\
	c & = (1-\alpha_1 \alpha_2 \alpha_3)^2.
\end{align*}

\section{Asymptotic results} \label{sec:asymptotic}

\subsection{Discussion}

For our asymptotic results we restrict to the special
case from Theorem \ref{thm:theorem1}, namely $\alpha_j = \alpha > 0$ 
for every $j$ with parameters $L= 4n$ and $K=2n$.
This choice of parameters corresponds to tilings of
a regular hexagon of size $2n \times 2n \times 2n$. We are motivated by
the work \cite{charlier2019periodic} where this tiling problem
was studied, but not with MVOPs. 
Instead it relied on monic scalar orthogonal 
polynomials $(p_{k,n})_k$  with orthogonality
\begin{equation} \label{eq:scalarpkn} 
	\frac{1}{2\pi i} \oint_{\gamma}
	p_{k,n}(z) z^j \frac{(z+1)^n (z+\alpha)^n}{z^{2n}} dz =
	0, \qquad j=0,1, \ldots, k-1. 
	\end{equation}
Note that we use a second subscript $n$ in order to emphasize
that the orthogonality weight is varying with $n$. 
		
As we saw in \cref{thm:theorem1} the point of
view of matrix orthogonality leads to monic scalar orthogonal
polynomials $q_k = q_{k,n}$ (if they exist) with
\begin{equation} \label{eq:scalarqkn} 
	\frac{1}{2\pi i} \oint_{\gamma_R}
	q_{k,n}(z) z^j \frac{(z+\alpha + \beta)^{2n}}{(z^2-\beta^2)^n}
	dz = 0, \qquad j=0,1, \ldots, k-1.
	\end{equation}
with $\beta = (1-\alpha)/2$ and $R > \beta$.
In the situation of \cref{thm:theorem1}(a),
$q_{2k}$ is equal to $q_{2k,2n}$, and $q_{2k+1}$
is equal to $q_{2k+1,2n} + c q_{2k,2n}$ for some constant $c$, 
provided that $q_{2k+1}$ has degree $2k+1$.

Both families of polynomials $p_{k,n}$ and $q_{k,n}$ can be used  for 
the analysis of the hexagon tiling model, but curiously enough it is not 
possible to map one family  directly to the other. However, we do see
similar behavior in the asymptotic behavior of $p_{n,n}$ and $q_{n,n}$ 
as $n \to \infty$.

The asymptotic behavior of $p_{n,n}$ is essentially done in 
\cite{charlier2019periodic}, but since the focus
there was on the hexagon tiling model, the results for the polynomials were not 
stated  explicitly in \cite{charlier2019periodic}. Let us
summarize however what we have. First of all we may restrict to $0 < \alpha \leq 1$,
because there is a symmetry $\alpha \mapsto 1/\alpha$ in
the tiling model.
This is reflected in the orthogonality \eqref{eq:scalarpkn}.
Indeed, if we use $p_{k,n}(z; \alpha)$ to denote the dependence
on $\alpha$, then we obtain from a simple
rescaling $z \mapsto \alpha z$ in \eqref{eq:scalarpkn} that
\[ p_{k,n}(\alpha z; \alpha)  = \alpha^k p_{k,n}(z; 1/\alpha). \]
We similarly find from \eqref{eq:scalarqkn} that
\[ q_{k,n}(\alpha z; \alpha) = \alpha^k q_{k,n}(z; 1/\alpha). \]

Next, it was found in \cite{charlier2019periodic} that
there is a phase transition in the hexagon tiling model 
at the critical value $\alpha = 1/9$. In the large
$n$ limit there are frozen regions where the tiling
is fixed with very high probability, and liquid regions where
one sees all three types of tiles in a random fashion.
For $0 < \alpha <  \frac{1}{9}$ the liquid region consists
of two disjoint ellipses, which merge at the critical
value of $\alpha$. For $\frac{1}{9} < \alpha \leq 1$
the liquid region is simply connected. See \cref{fig:low alpha tiling} below
for a hexagon tiling with $\alpha < \frac{1}{9}$. 
The transition is reflected in the behavior of the
zeros of the polynomials $(p_{n,n})_n$ as $n \to \infty$.
For $\frac{1}{9} < \alpha < 1$ the zeros tend to 
the circular arc $|z| = \sqrt{\alpha}$, $|\arg z| < \theta_{\alpha}$, 
for a certain $\theta_{\alpha} < \pi$, while for $0 < \alpha < \frac{1}{9}$ they 
tend to a closed contour as $n \to \infty$.

\begin{figure}[t] \vspace*{-2cm}
	\centering
\begin{overpic}[scale=0.5]{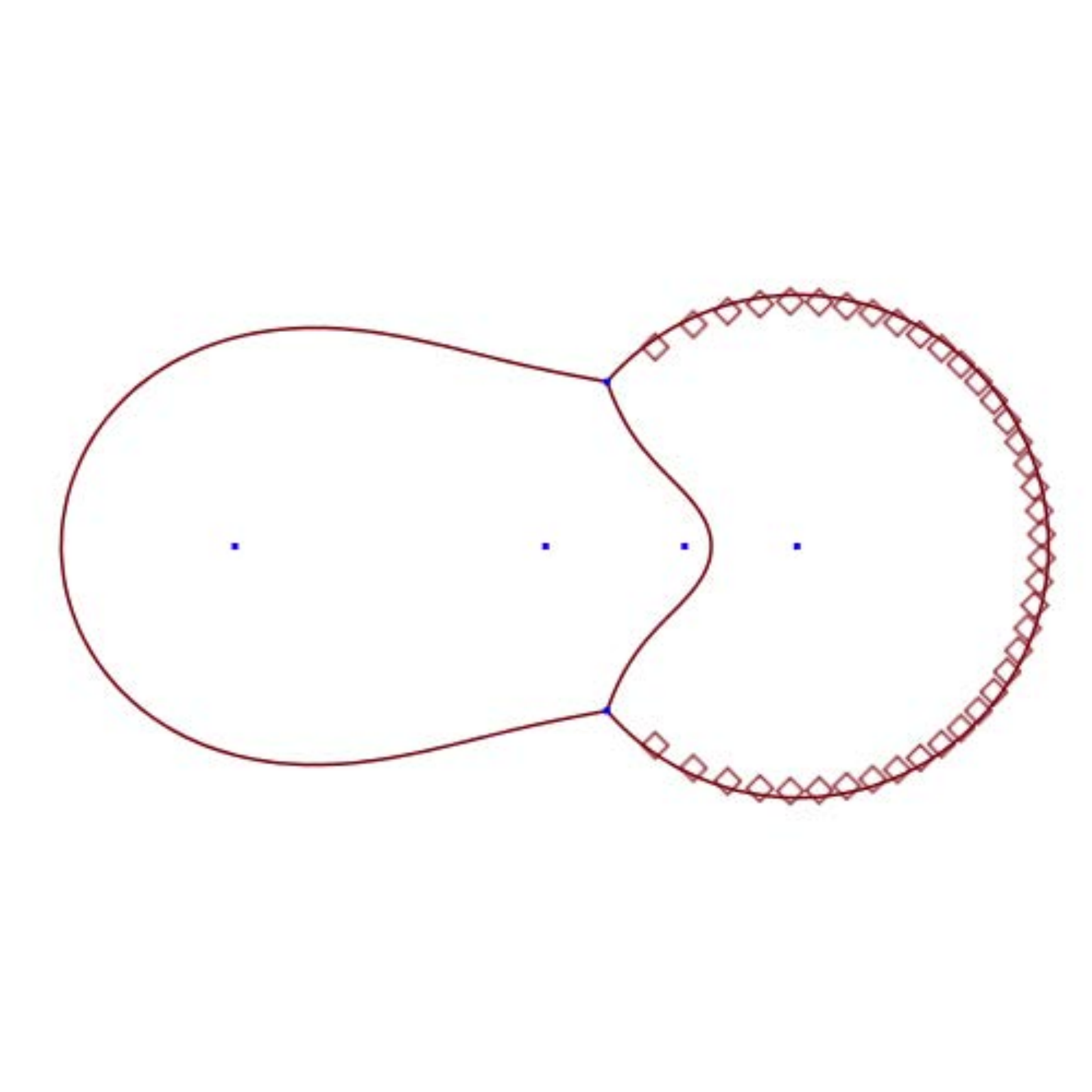}
	\put(18,46){$-1$}
	\put(44,46){$-\sqrt{\alpha}$}
	\put(58,46){$-\alpha$}
	\put(72,46){$0$}
	\put(52,67){$\widehat{z}_+$}
	\put(53,30){$\widehat{z}_-$}
\end{overpic} \vspace*{-2cm} 

\caption{\label{fig:zeros40CDKL} Zeros of the polynomial $p_{n,n}$
with $n=40$ and $\alpha=0.2$, together with the critical trajectories 
of $\widehat{Q}_{\alpha}(z) d z^2$. The zeros tend to one
of the critical trajectories as $n \to \infty$, namely to 
the circular arc connecting $\widehat{z}_-$ and $\widehat{z}_+$.}
\end{figure}
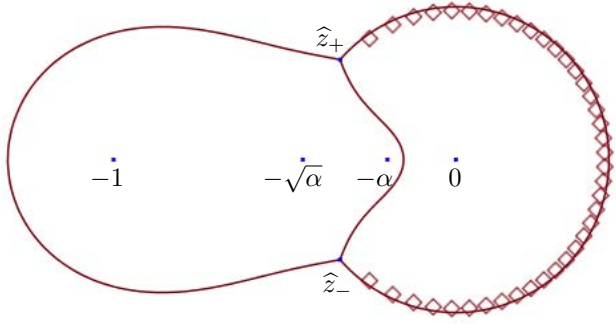

By \cite[Lemma 4.5]{charlier2019periodic}, the circular arc is a critical 
trajectory of the quadratic differential $\widehat{Q}_{\alpha}(z) dz^2$ where
\begin{equation} \label{eq:Qalphatilde} 
	\widehat{Q}_{\alpha}(z) =
	\frac{(z-\widehat{z}_+)(z- \widehat{z}_-) (z+\sqrt{\alpha})^2}{z^2(z+1)^2(z+\alpha)^2} \end{equation}
with
\[ \widehat{z}_{\pm} 
	= - \frac{3 - 2 \sqrt{\alpha} + 3 \alpha}{8}
		\pm \frac{3i(1+\sqrt{\alpha})}{8}
			\sqrt{\left(1- \tfrac{\sqrt{\alpha}}{3}\right)
				\left(3 \sqrt{\alpha} -1\right) }. \]
There are three critical trajectories connecting
the two simple zeros $\widehat{z}_{\pm}$ of the quadratic
differentials. One of these is the circular arc that
attracts the zeros of $p_{n,n}$ as $n \to \infty$
in the case $1/9 < \alpha \leq 1$.
A brief overview of basic properties of quadratic
differential is given in 
\cite[Section 4]{mf_rakhmanov_critical_2011}, see also 
\cite{pommerenke_univalent_1975} and \cite{strebel_quadratic_1984} for more extensive accounts.

\subsection{Zeros of $q_{n,n}$}
For the asymptotic results on the polynomials $q_{n,n}$ in this paper we
also restrict to $1/9 < \alpha \leq 1$.

The quadratic differential $Q_{\alpha}(z) dz^2$
that is relevant for the zeros of $q_{n,n}$ is given in
the next definition.
\begin{definition}\label{def:rational function Q}
	For $\frac19 < \alpha \leq 1$, we define the two points $z_\pm = z_\pm(\alpha)$ as
	\begin{align}\label{simpleZerosRationalFunctionQ}
	z_\pm = z_\pm(\alpha) = (1+\sqrt{\alpha}) \left( - \frac{1+\sqrt{\alpha}}{8} \pm \frac{3 i}{8} \sqrt{\left(1-\frac{\sqrt{\alpha}}{3}\right) 
		\left(3 \sqrt{\alpha} - 1\right)} \right) 
	\end{align}	
	and the rational function $Q = Q_\alpha$ as
	\begin{align}\label{rationalFunctionQ}
	Q_{\alpha}(z) = \frac{(z - \zplus) (z - \zminus) 
		\left(z + \frac{(1-\sqrt{\alpha})^2}{2}\right)^2 }{(z - \beta)^2 (z + \beta)^2 (z + \alpha + \beta)^2}
	\end{align}
	with $\beta = \frac{1-\alpha}{2}$ as before.
\end{definition}

A little calculation shows that
\begin{equation} \label{eq:QvsQtilde} 
	Q_{\alpha}(-z-\alpha-\beta) = \widehat{Q}_{\alpha}(z) 
	\end{equation}
with $\widehat{Q}_{\alpha}$ given by \eqref{eq:Qalphatilde},
and therefore the critical trajectories of $Q_{\alpha} dz^2$ are mapped 
to those of $\widehat{Q}_{\alpha} dz^2$ by the mapping $z \mapsto -z-\alpha - \beta$.
The latter ones where determined in \cite{charlier2019periodic}.
We then have the following.

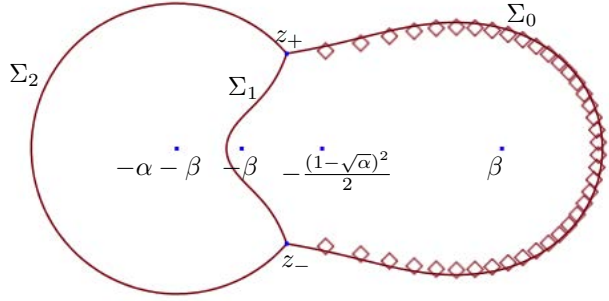
\begin{figure}[t] \vspace*{-2cm}
	\centering
	\begin{overpic}[scale=0.5]{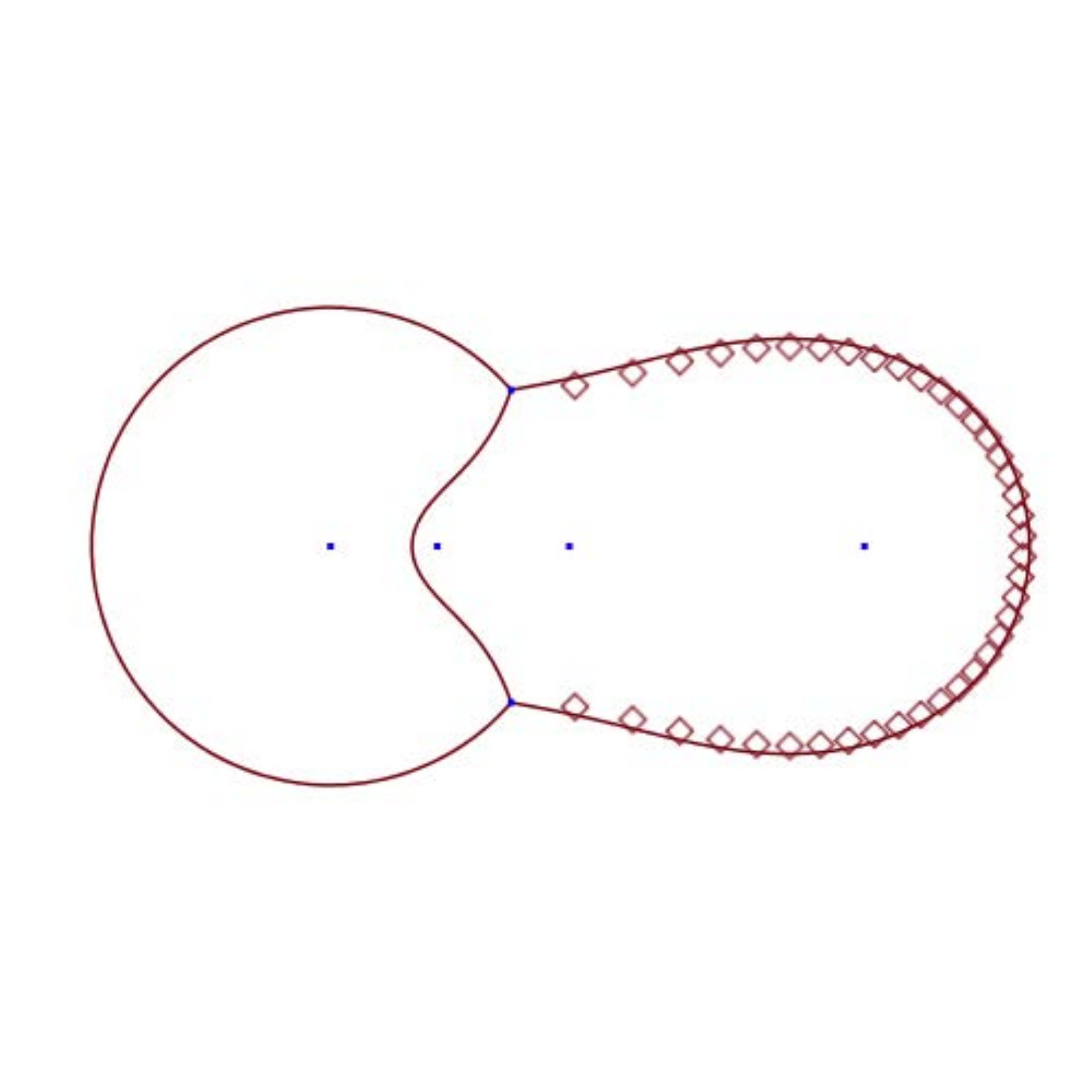} 
		\put(21,46){$-\alpha-\beta$}
		\put(37,46){$-\beta$}
		\put(46,45){$- \frac{(1-\sqrt{\alpha})^2}{2}$}
		\put(77,46){$\beta$}
		\put(45,66){$z_+$}
		\put(46,32){$z_-$} 
		\put(80,69){$\Sigma_0$}
		\put(38,58){$\Sigma_1$}
		\put(5,60){$\Sigma_2$}
		\end{overpic} \vspace*{-2cm}
	
	\caption{\label{fig:zeros40GK} Zeros of the polynomial $q_{n,n}$
		with $n=40$ and $\alpha=0.2$, together with the critical trajectories of $Q_{\alpha}(z) d z^2$. The zeros tend to $\Sigma_0$, which is one
		of the critical trajectories as $n \to \infty$, but
		it is \emph{not} the 
		circular arc $\Sigma_2$ connecting $z_-$ and $z_+$. The figure also shows
	the zeros and poles of the quadratic differential.}
\end{figure}

\begin{lemma} \label{lem:zerosqnn}
		Let $\frac{1}{9} < \alpha \leq 1$.
		\begin{enumerate}
	\item[\rm (a)] 
		There are three critical trajectories $\Sigma_0$, $\Sigma_1$,
		$\Sigma_2$ of $Q_{\alpha}(z) dz^2$
		that connect the simple zeros $z_{\pm}$ of $Q_{\alpha}$,
		as shown in  \cref{fig:zeros40GK}.
	\item[\rm (b)] 
	Each trajectory $\Sigma_j$ carries a probability measure $\mu_j$ 	given by
	\begin{align} \label{def:mu}
	d \mu_j(s) = \frac{1}{\pi i} Q_{\alpha}(s)^{1/2} d s, \quad s \in \Sigma_j, \quad j=0,1,2,
	\end{align}
	where $d s$ denotes the complex line element and the 
	appropriate square root is taken so that \eqref{def:mu}
	is a positive measure.
	\end{enumerate}
\end{lemma}
\begin{proof} 
	Part (a) is immediate from \eqref{eq:QvsQtilde} and 
	the corresponding result in \cite[Lemma 4.6]{charlier2019periodic} 
	on the critical trajectories of $\widehat{Q}_{\alpha} dz^2$.
	The proof of part (b) is analogous to the proof
	of \cite[Proposition 4.4]{charlier2019periodic}.
\end{proof}

The asymptotic behavior of the zeros of the polynomials $(q_{n,n})_n$
follows from a strong asymptotic formula. The formula involves the $g$-function
\begin{align} \label{eq:gzdef}
g(z) & = \int \log(z-s) d\mu_0(s), 
\end{align} 
associated with the measure $\mu_0$, which 
is defined and analytic in $\mathbb C \setminus
(\Sigma_0 \cup (-\infty, x^*)])$ where $x^*$ denotes the point
of intersection of $\Sigma_0$ with the positive real axis,
as well as
the following functions that are determined by the endpoints
$z_{\pm}$, see  \eqref{simpleZerosRationalFunctionQ}, of $\Sigma_0$,
\begin{align} 
\label{eq:A0z}
A_0(z) & = \frac{1}{2}(a(z) + a(z)^{-1}), \\		
\label{eq:az}
a(z) & = \left( \frac{z-z_+}{z-z_-} \right)^{1/4}, \\
\label{eq:psiz}
\psi(z) & = \frac{1}{2} \left( z -  \frac{z_+-z_-}{2}
+ \left((z-z_+)(z-z_-)\right)^{1/2} \right).
\end{align}
The fractional
powers in \eqref{eq:az} and \eqref{eq:psiz} are defined
and analytic in $\mathbb C \setminus \Sigma_0$ while being
real and positive for large positive real $z$.

\begin{proposition} \label{prop:zerosqnn}	
	Let $\frac{1}{9} < \alpha \leq 1$. With the notation above we then have.
	\begin{enumerate}
	\item[\rm (a)] For each fixed $k \in \mathbb Z$, the
	polynomial  $q_{n+k,n}$ exists for $n$ large enough,
	and 
	\begin{equation} \label{eq:qnkasymptotics} 
		q_{n+k,n}(z) = A_0(z) \psi^k(z) e^{ng(z)}
	\left(1+ \bigO\left(\frac{1}{n(1+|z|)}\right)\right), 
	\quad z \in \mathbb C \setminus \Sigma_0, \end{equation}
	where the $\bigO$ is uniform for $z$ in compact
	subsets of $\overline{\mathbb C} \setminus \Sigma_0$. 

	\item[\rm (b)] 
	The zeros of $q_{n,n}$ tend to $\Sigma_0$
	as $n \to \infty$, and $\mu_0$ is the limit of
	the  normalized zero counting measures.
	\end{enumerate}
\end{proposition}	
We find it remarkable that the zeros of $q_{n,n}$ tend to $\Sigma_0$,
but $\Sigma_0$ is not the image of the circular arc under the mapping $z \mapsto -z-\alpha - \beta$,
that attracts the zeros of $p_{n,n}$ as $n \to \infty$, compare
the Figures \ref{fig:zeros40CDKL} and \ref{fig:zeros40GK}. Thus the
two sequences of polynomials are genuinely different.

The normalized zero counting measure $\nu(q)$ of a polynomial $q$ 
of degree $n$ is the measure that assigns mass $1/n$ to each of
its zeros, where zeros are counted according to their multiplicities, i.e.,
\begin{align} \label{eq:zero counting measure}
	\nu(q) = \frac{1}{n} \sum_{q(x) = 0} \delta_x. 
\end{align} The convergence in part (b) of \cref{prop:zerosqnn} is in the sense 
of weak convergence
of measures and so $\nu(q_{n,n}) \to \mu_0$ weakly means that
\[ \lim_{n \to \infty} \int f d\nu(q_{n,n}) = \int f d\mu_0 \]
for every bounded continuous function on $\mathbb C$.

The study of zero distributions of polynomials $q_{n,n}$ 
with respect to a varying non-Hermitian orthogonality weight has a
long history, starting with the pioneering works of 
Gonchar and Rakhmanov \cite{gonchar_equilibrium_1989}.
Under quite general conditions, the zeros of polynomials that satisfy a non-Hermitian orthogonality 
tend to smooth curves, that are  critical trajectories of quadratic differentials, 
see for example \cite{deano_large_2014, huybrechs_zero_2014, martinez-finkelshtein_silva}.
See also \cite{martinez-finkelshtein_rakhmanov_dream} and the many references cited therein.
 
\subsection{Zeros of $\det P_{n,n}$}
Our main interest is in the MVOP $P_{n,n}$ that has the orthogonality
\eqref{eq:MVOP1} with matrix weight \eqref{eq:matrixW1}
\[ W(z) = \frac{W_{\alpha}(z)^{4n}}{z^{2n}}, \]
since we recall that $L=4n$ and $K=2n$.
Our second main result deals with the limiting behavior
of the zeros of the determinant of $P_{n,n}$. 
These are sometimes also simply called zeros of $P_{n,n}$ in the literature, 
see for example \cite[p. 98]{duran_lopez-rodriguez-MOPs}.

In the literature there are already many results on the zeros of MVOP when the matrix orthogonality is
on the real line with a semi-positive definite weight matrix.
Theorem 1.1 in \cite{duran_lopez-rodriguez-MOPs} states some properties of the 
zeros of MVOP of finite degree, and the limiting behavior of the zeros of MVOP is 
investigated from two perspectives by Dur\'{a}n, L\'{o}pez-Rodriguez and Saff in \cite{duran_zero_1999}.
See also Theorem 5.2 in \cite{damanik}.
The results in \cite{duran_zero_1999} are  generalized by Delvaux and Dette in \cite{delvaux_dette}.
In the present situation, the limiting behavior of the zeros of $\det P_{n,n}$ is as follows.

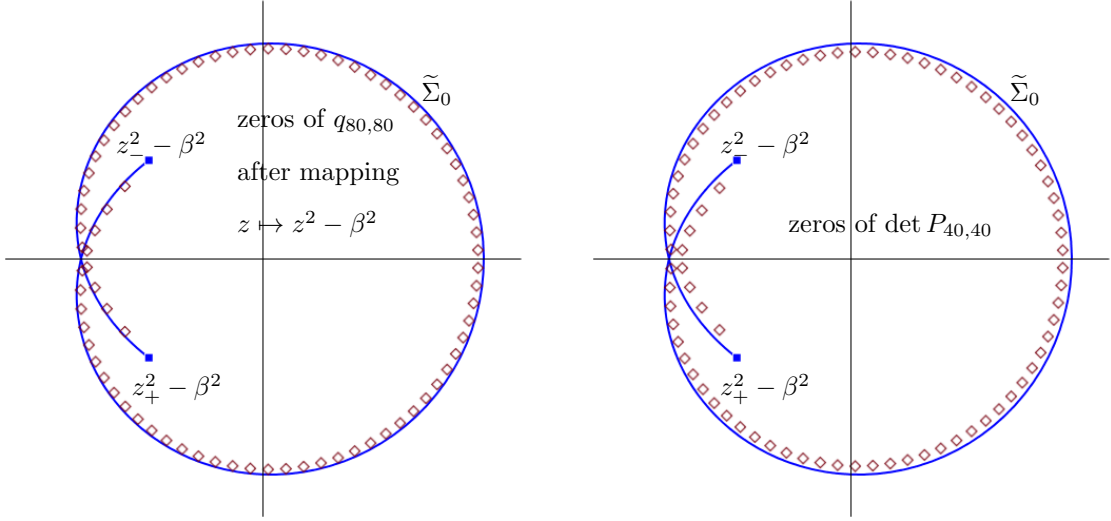
\begin{figure}[t] \vspace*{-1cm}
	\centering
	\begin{overpic}[height=7cm,keepaspectratio]{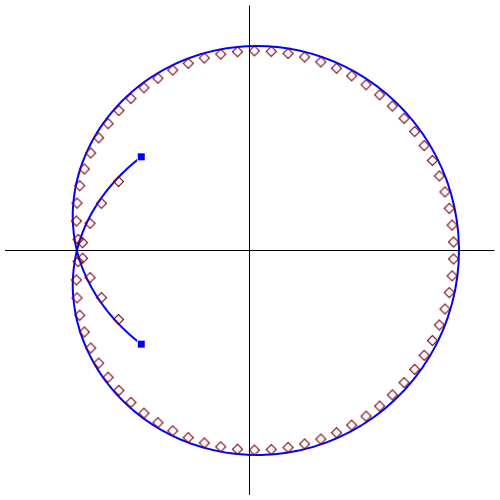}
		\put(80,80){$\widetilde{\Sigma}_0$}
		\put(45,75){zeros of $q_{80,80}$}
		\put(45,65){after mapping}
		\put(45,55){$z \mapsto z^2-\beta^2$}
		\put(25,24){$z_+^2-\beta^2$}
		\put(22,70){$z_-^2-\beta^2$}
		\end{overpic}
	\qquad
	\begin{overpic}[height=7cm,keepaspectratio]{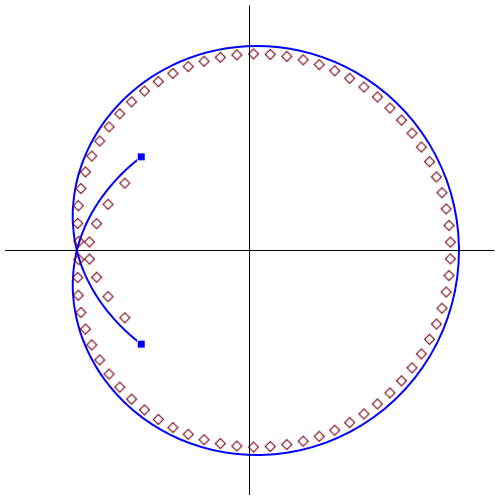}
		\put(80,80){$\widetilde{\Sigma}_0$}
		\put(38,55){zeros of $\det P_{40,40}$}
		\put(25,24){$z_+^2-\beta^2$}
		\put(25,70){$z_-^2-\beta^2$}
		\end{overpic} \vspace*{0cm}
	
	\caption{\label{fig:zeros of determinant}
		The curve $\widetilde{\Sigma}_0$ and
		the zeros of $q_{80,80}$ after the transformation $z \mapsto z^2 - \beta^2$ (left) and the zeros of $\det P_{40,40}$ (right).}
\end{figure}

\begin{theorem}\label{thm:zero distribution determinant}
	Let $\frac{1}{9} < \alpha \leq 1$. Then
	the zeros of $\det P_{n,n}$ tend to the curve
	\begin{equation} \label{eq:pushSigma} 
	\widetilde\Sigma_0 = \{ z^2 - \beta^2 \mid z \in \Sigma_0 \} \end{equation}
	as $n \to \infty$. 
	Furthermore, the normalized zero counting measures of 
	$\det P_{n,n}(z)$ tend to the probability measure $\widetilde\mu_0$ that is 
	the pushforward of $\mu_0$ under the map $z \mapsto z^2 - \beta^2$.
\end{theorem}

The pushforward measure $\widetilde{\mu}_0$ is characterized by the property that
\begin{equation} \label{eq:pushmu0}	
\int f d\widetilde{\mu}_0 = \int f(z^2 - \beta^2) d\mu_0(z) \end{equation}
for every continuous function $f$ on $\widetilde{\Sigma}_0$.

From \cref{prop:zerosqnn}(b) it follows that after transformation
$z \mapsto z^2 - \beta^2$, the zeros of 
$q_{n,n}(z)$ tend to $\widetilde{\Sigma}_0$
as $n\to \infty$ with $\widetilde{\mu}_0$ as limiting normalized
zero counting measure.   The zeros of $\det P_{n,n}$
have the same limiting behavior, as also shown in 
\cref{fig:zeros of determinant},
where we compare the zeros of $\det P_{n,n}$ with
those of $q_{2n,2n}$ for $n=40$. The similar behavior is to 
be expected, because of the identity 
\begin{equation} \label{eq:identity_determinant} 
\det P_{n,n}(z) = 
\frac{q_{2n,2n}(\zeta)q_{2n+1,2n}(-\zeta) - q_{2n+1,2n}(\zeta) q_{2n,2n}(-\zeta)}{2\zeta},
\qquad z = \zeta^2 - \beta^2, \end{equation} 
that easily follows from \eqref{eq:scalarOP2}.

\subsection{Zeros of the $(1,2)$-entry of $P_{n,n}$}

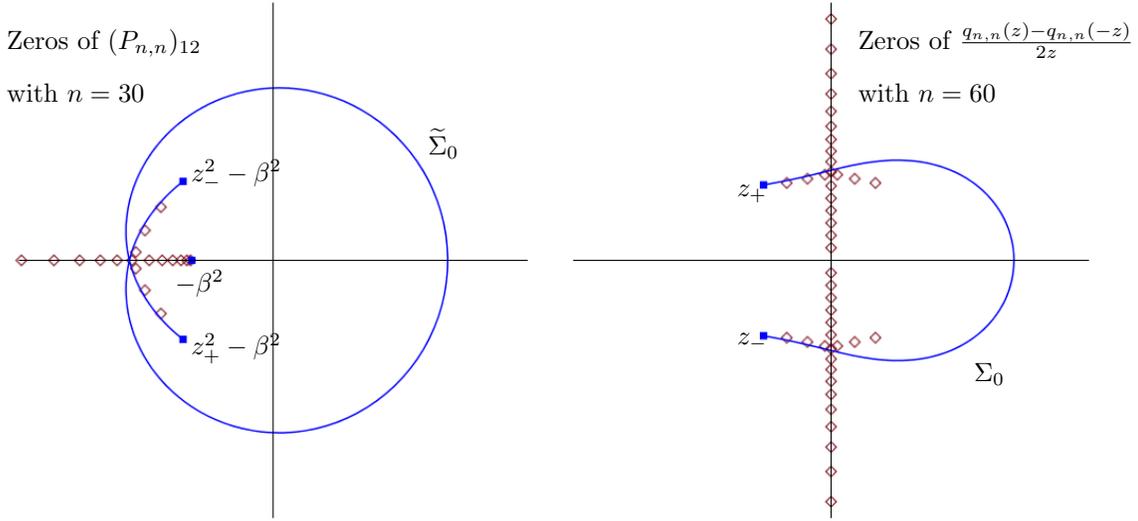
\begin{figure}[t]  \vspace*{-1cm}
	\centering
	\begin{overpic}[height=7cm, keepaspectratio]{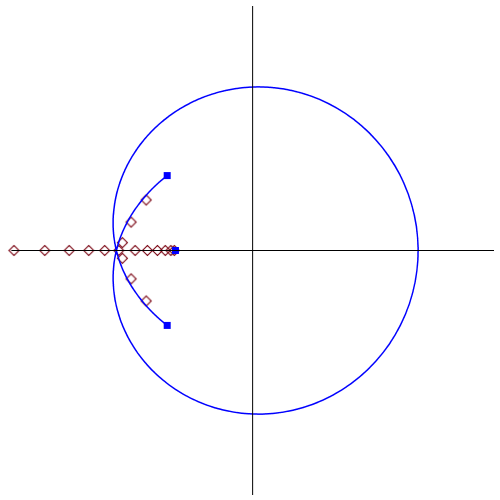}
		\put(80,70){$\widetilde{\Sigma}_0$}
		\put(35,65){$z_-^2-\beta^2$}
		\put(35,32){$z_+^2-\beta^2$}
		\put(32,44){$-\beta^2$}
		\put(0,90){Zeros of $(P_{n,n})_{12}$}
		\put(0,80){with $n=30$}
	\end{overpic}
\quad
	\begin{overpic}[height=7cm, keepaspectratio]{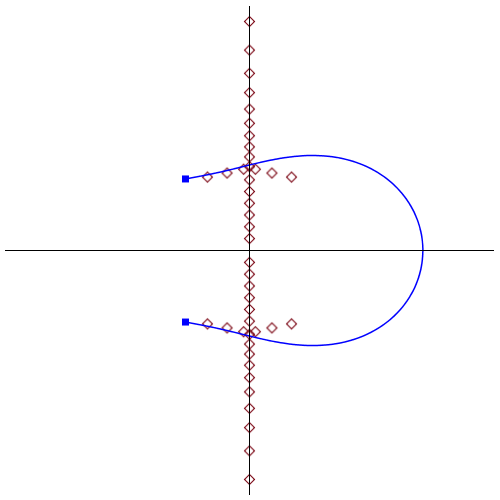}
	\put(77,27){$\Sigma_0$}
	\put(32,62){$z_+$}
	\put(32,34){$z_-$}
	\put(55,90){Zeros of $\frac{q_{n,n}(z)-q_{n,n}(-z)}{2z}$}
	\put(55,80){with $n=60$}
\end{overpic}
\caption{\label{fig:zerosPnn12} 
	The left panel shows the contour $\tilde{\Sigma}_0$
	and part of the zeros of $(P_{n,n})_{12}$ with $n=30$
	and $\alpha = 0.2$. There are more zeros on the negative
	real axis. As $n \to \infty$ the real zeros fill out
	the interval $(-\infty,-\beta^2]$ and the non-real zeros
	follow part of the contour $\widetilde{\Sigma}_0$. 
	The right panel shows the contour $\Sigma_0$ and part of
	the zeros of
	$\frac{q_{n,n}(z)-q_{n,n}(-z)}{2z}$ for the values $n=60$
	and $\alpha =0.2$. There are more zeros on the imaginary
	axis. As $n \to \infty$ the zeros fill out
	the full imaginary axis. Moreover, they will accumulate on the part of $\Sigma_0$ in
	the left half-plane and on the part of $-\Sigma_0$
	in the right half-plane.}
\end{figure}

Finally, we study the asymptotic distribution of the zeros
of the top right entry of the matrix $P_{n,n}$, see \cref{fig:zerosPnn12}. 
The individual entries of a matrix valued orthogonal polynomial 
are not a natural quantity to study, because they depend on the normalization of the 
matrix $P_{n,n}$, which in our case is taken to be monic.
Nevertheless, we were curious to look for the limiting distribution of their zeros, 
and we chose the $(1,2)$-entry
since this entry has the simplest expression in terms of the
scalar orthogonal polynomials. 
Indeed, by \eqref{eq:scalarOP2}, the $(1,2)$-entry of $P_{n,n}$ is given by
\begin{equation} \label{eq:Pnn12}
 (P_{n,n})_{12}(z) =  
\frac{q_{2n,2n}(\zeta) - q_{2n,2n}(-\zeta)}{2 \zeta},
\end{equation}
where $\zeta = (z + \beta^2)^{1/2}$ as before.
We find the limiting distribution of the zeros of
the right-hand side of \eqref{eq:Pnn12} and then
the  limiting distribution of the zeros of $(P_{n,n})_{12}$
follows after a coordinate transformation. Our numerical explorations
show that the other entries of $P_{n,n}$ have the same limiting
behavior of their zeros.

In order to prove the result we need the inequality \eqref{eq:inequality g-function} (see below)
for the  $g$-function \eqref{eq:gzdef} associated with
the measure $\mu_0$ on $\Sigma_0$.
We were not able to prove \eqref{eq:inequality g-function} analytically, 
but we are able to establish it under a geometric 
condition on $\Sigma_0$, see  \cref{thm:geometrical condition}
which is supported by numerical evidence.

In the statement of \cref{thm:zero distribution top right entry}
and also further throughout the paper, we use 
$\LHP = \{z \in \mathbb C \mid \Re z < 0 \}$ and 
$\RHP = \{z \in \mathbb C \mid \Re z > 0 \}$ 
to denote the open left and right half-planes, respectively.

\begin{theorem}\label{thm:zero distribution top right entry}
	Let $\frac19 < \alpha \leq 1$. 	Suppose that
	\begin{align}\label{eq:inequality g-function}
	\Re g(-z) > \Re g(z), \quad \text{for } z \in \RHP.
	\end{align}
	Then the sequence of normalized zero counting measures of
	the polynomials $\frac{q_{n,n}(z) - q_{n,n}(-z)}{2z}$
		tends weakly to the probability measure 
	\begin{equation} \label{eq:limitnu} 
		\nu = \nu_L + \nu_R + \nu_0, \end{equation}
	as $n \to \infty$, 
	where $\nu_L = \mu_0|_{\LHP}$, $\nu_R = \mu^*_0|_{\RHP}$,
	and $\nu_0 = \Bal\left(\mu_0|_{\RHP}\right) - \Bal\left(\mu_0|_{\LHP}\right)$. 
	Here $\Bal$ denotes the balayage of a measure to the imaginary axis, and
	$\mu^*_0$ is the pushforward of the measure $\mu_0$ under
	the sign change $z \mapsto -z$.  
	
	The sequence of normalized zero counting measures
	of $(P_{n,n})_{12}$ tends to $\widetilde{\nu}$
	where $\widetilde{\nu}$ is the pushforward of
	$\nu$ under the map $z \mapsto z^2-\beta^2$.
\end{theorem}

The balayage measure $\hat{\mu} = \Bal(\mu)$ of a measure $\mu$ onto the
imaginary axis is characterized by the properties 
that $\supp(\hat{\mu}) \subset i \mathbb R$, $\int d\hat{\mu} = \int d\mu$ and 
$U^{\hat{\mu}}(z) = U^{\mu}(z)$ for $z \in i \mathbb R$
where $U^{\mu}(z) = \int \log \frac{1}{|z-s|}d\mu(s)$ denotes
the logarithmic potential of $\mu$, see e.g.\ \cite{saff_logarithmic_1997}.

In \eqref{eq:limitnu} we have that $\nu_0$ is the part of $\nu$ that
is on the imaginary axis. It is given as the difference between
two balayage measures, which a priori need not be positive.  
The condition \eqref{eq:inequality g-function} is needed
in order to show that 
\begin{equation} \label{eq:Bal inequality}
\Bal\left(\mu_0|_{\LHP}\right) \leq \Bal\left(\mu_0|_{\RHP}\right)
\end{equation}
in the sense of measures, such that $\nu_0$ is indeed positive,
see \cref{lem:limit signed measure is a probability measure}.
Then according to \cref{thm:zero distribution top right entry} 
$\nu_0$ is the limiting distribution for the zeros 
on the imaginary axis. Very loosely speaking, the inequality
\eqref{eq:Bal inequality} expresses
the fact that there is more of $\mu_0$ in the right half-plane than 
in the left half-plane.

\begin{figure}[t]\vspace*{0cm}
	\centering
	\begin{overpic}[scale=0.4]{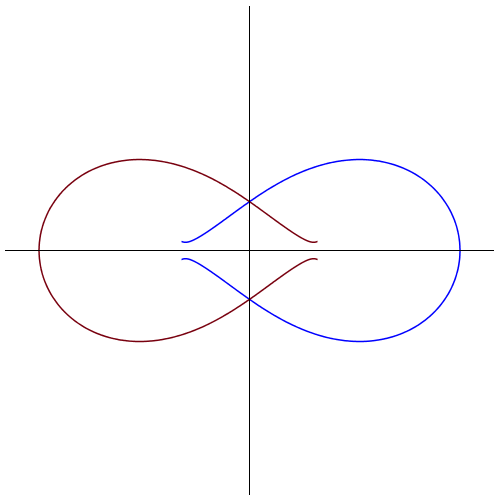} 	\put(90,60){$\Sigma_0$}
	\put(0,60){$-\Sigma_0$}
		\put(20,55){$\Omega_0$}
		\put(5,90){$\alpha = 0.112$}
		\put(61,52){$-z_-$}
		\put(61,45){$-z_+$}
		\put(32,53){$z_+$}
		\put(32,45){$z_-$}
	\end{overpic}
	\qquad
		\begin{overpic}[scale=0.4]{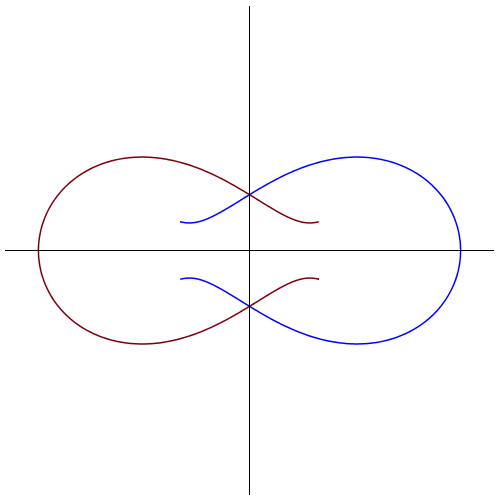} 
		\put(0,60){$-\Sigma_0$}	
		\put(20,55){$\Omega_0$}
		\put(5,90){$\alpha = 0.12$}
		\put(61,52){$-z_-$}
		\put(61,45){$-z_+$}
		\put(32,53){$z_+$}
		\put(32,45){$z_-$}
		\end{overpic} \\  \vspace*{-1cm}
	
		\begin{overpic}[scale=0.4]{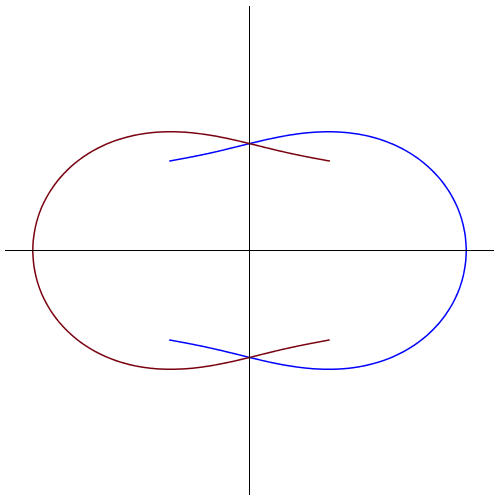} 
	\put(83,70){$\Sigma_0$}
	\put(7,70){$-\Sigma_0$}	
	\put(20,55){$\Omega_0$}
	\put(5,90){$\alpha = 0.2$}	
	\put(61,65){$-z_-$}
	\put(61,33){$-z_+$}
	\put(31,65){$z_+$}
	\put(31,33){$z_-$}
\end{overpic}
	\qquad
		\begin{overpic}[scale=0.4]{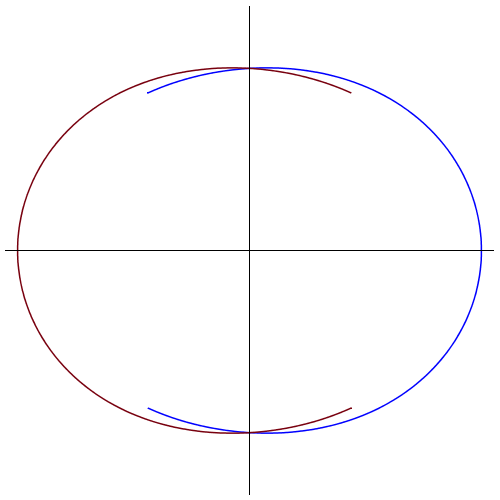} 
	\put(86,75){$\Sigma_0$}
	\put(4,75){$-\Sigma_0$} 	
	\put(20,55){$\Omega_0$}
	\put(5,90){$\alpha = 0.4$}	
	\put(64,77){$-z_-$}
	\put(65,21){$-z_+$}
	\put(27,77){$z_+$}
	\put(27,21){$z_-$}
	\end{overpic}
	\caption{\label{fig:geometrical assumption}
	The contours $\Sigma_0$ (in blue) and $-\Sigma_0$ (in red) for the values of $\alpha$ 
	equal to $0.112$, $0.12$, $0.2$ and $0.4$ from left to right.
	The part of $\Sigma_0$ in the right half-plane and the part of $-\Sigma_0$ in the left half-plane enclose a bounded domain $\Omega_0$
	and the figures show that the part of $\Sigma_0$ in the left half-plane
	belongs to $\Omega_0$.
	}
\end{figure}
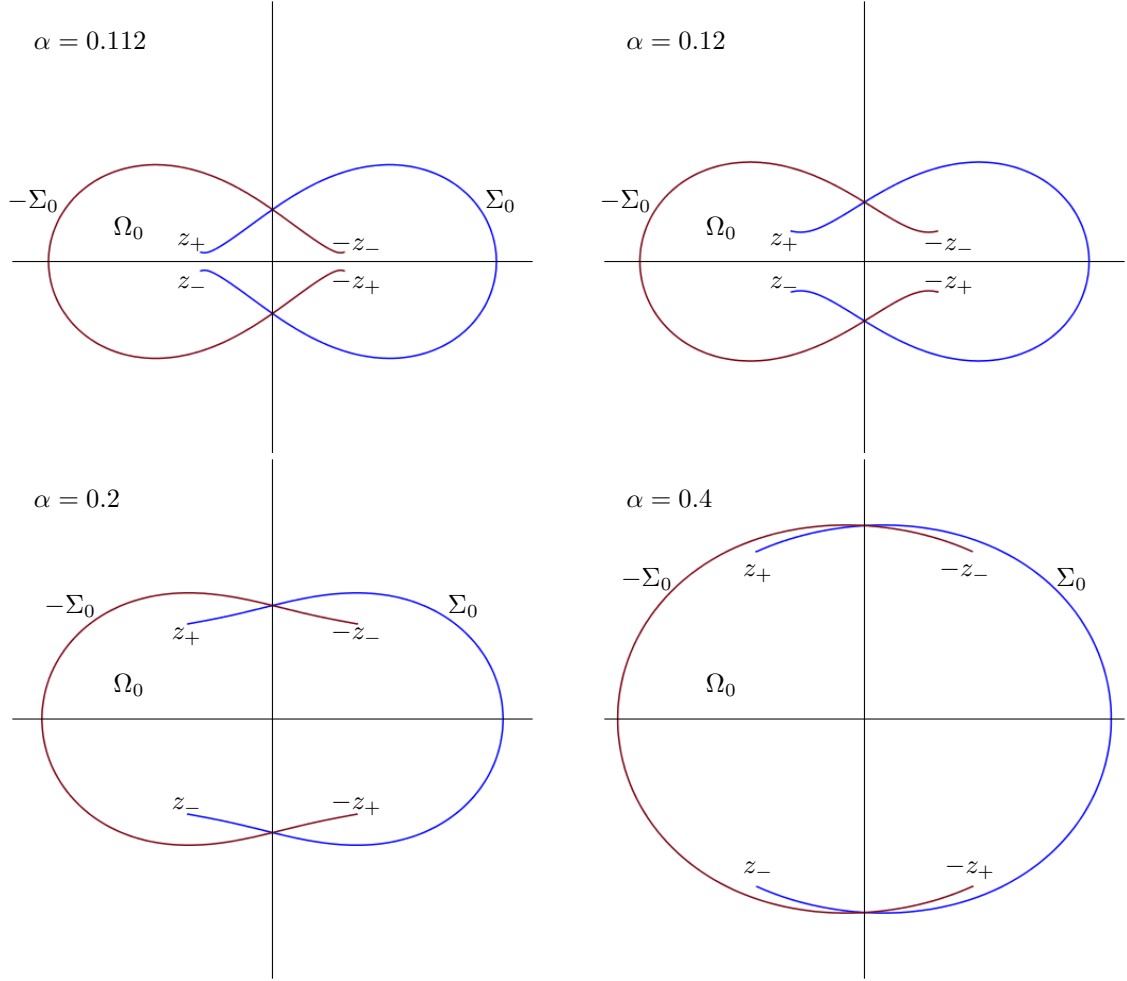

\subsection{Geometric condition}
The trajectory $\Sigma_0$ intersects the imaginary axis
and so it is partly in the right half-plane and partly in
the lower half-plane. The part of $\Sigma_0$ in the right half-plane together with the part of $-\Sigma_0$
in the left half-plane enclose a bounded domain $\Omega_0$,
as shown in \cref{fig:geometrical assumption}.

\begin{theorem} \label{thm:geometrical condition}
	Suppose that the part of $\Sigma_0$ in the left half-plane
	(and by symmetry, the part of $-\Sigma_0$ in the right half-plane) are contained in $\Omega_0$. Then the
	inequality \eqref{eq:inequality g-function} holds.
\end{theorem}

The above condition can be equivalently reformulated as a requirement on $\widetilde{\Sigma}_0$ from \eqref{eq:pushSigma} as follows.
The part of $\widetilde{\Sigma}_0$ that is the image of the part of $\Sigma_0$ in the closed right half-plane under the mapping $z \mapsto z^2 - \beta^2$ is a closed curve that divides the plane in a bounded and an unbounded set: the requirement then becomes that the remaining part of $\widetilde{\Sigma}_0$ should be contained in the bounded set.
See for example \cref{fig:zeros of determinant}.

The condition of \cref{thm:geometrical condition} is
strongly supported by computational evidence. See \cref{fig:geometrical assumption} for plots of $\Sigma_0$, $-\Sigma_0$ and $\Omega_0$ for
a number of $\alpha$ values.

\subsection{Overview of the rest of the paper}

The rest of the paper is organized as follows. 
In \cref{sec:tilings}, we discuss lozenge tilings of the hexagon
and its connection with MVOPs and this leads to the proof of 
\cref{prop:exist}.
In \cref{sec:proof of main result}, we prove \cref{thm:theorem1}.
Then, in \cref{sec:strong asymptotics}, we derive the strong asymptotics of the scalar orthogonal polynomials $q_{n+k,n}$ as $n \to \infty$
using the Deift-Zhou steepest descent method with a small twist.
We use the strong asymptotics of $q_{n,n}$ and $q_{n+1,n}$ to prove \cref{prop:zerosqnn,thm:zero distribution determinant}.
In \cref{sec:proofofthm25}, we study the properties of the upper right entry of $P_n$ and prove \cref{thm:zero distribution top right entry}.
Finally, in \cref{sec:proofofthm26}, \cref{thm:geometrical condition} 
is proved.

\section{Tilings of a hexagon} \label{sec:tilings}

\subsection{Introduction} \label{subsec:tilings1}

Random lozenge tilings of a hexagon have been studied extensively 
in the last decades because of its remarkable connections
with various fields of mathematics and physics, see the book \cite{gorin_2021} of Vadim Gorin for an introduction to the topic.
In the simplest random model one assigns an equal probability
to each possible tiling, and within this model the arctic circle
phenomenon was observed and proved, see \cite{cohn_larsen_propp} or \cite[Section 3.4]{baik_discrete}. 
Periodically weighted tilings were studied
in \cite{kenyon_okounkov_sheffield} and in \cite{chhita_johansson} (for a related
tiling model of a so-called Aztec diamond). 
Recently in \cite{duits2017periodic}, a new technique
based on MVOPs was developed to study random tiling models with periodic weightings. The matrix valued orthogonality
\eqref{eq:MVOP1} with weights \eqref{eq:weightW} 
appears in this context, as we will explain now.

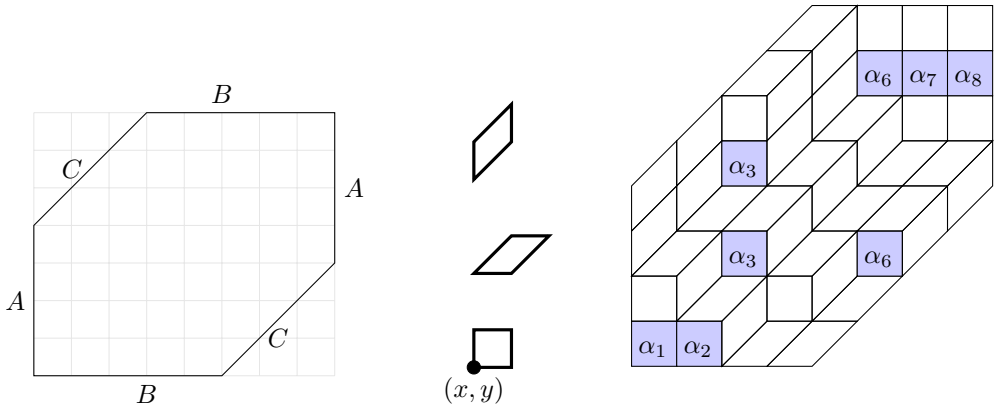
\begin{figure}[t]
	\centering
	\begin{tikzpicture}[scale=0.5]
	\draw[help lines, color=black!10] (0,0) grid (8,7);
	\draw (0,0) -- (5,0) -- (8,3) -- (8,7) -- (3,7) -- (0,4) -- cycle;
	\draw (-1,2) node[right] {$A$}; \draw(8,5) node[right]{$A$};
	\draw (3,0) node[below] {$B$}; \draw (5,7) node[above]{$B$};
	\draw (6.5,1.5) node[below]{$C$}; \draw (0.5,5.5) node[right]{$C$};
	\end{tikzpicture}
	\centering \qquad
	\begin{tikzpicture}[scale=0.5]
		\draw[draw=black, very thick] (0,0) rectangle (1,1);
		\filldraw (0,0) circle (5pt) node[below]{$(x,y)$};
		\draw[draw=black, very thick] (0,5) -- (1,6) -- (1,7) -- (0,6) -- cycle;
		\draw[draw=black, very thick] (0,2.5) -- (1,2.5) -- (2,3.5) -- (1,3.5) -- cycle;
	\end{tikzpicture} 
	\centering \quad
\begin{tikzpicture}[scale=0.6]
	
	\draw[white] (-1,-1) -- (9,9);
	
	\filldraw[draw=black, fill=blue!20] (0,0) -- (1,0) node[above left]{$\alpha_1$} -- (1,1) -- (0,1) -- cycle;
	\filldraw[draw=black, fill=blue!20] (1,0) -- (2,0) node[above left]{$\alpha_2$} -- (2,1) -- (1,1) -- cycle;
	\draw[draw=black] (2,0) -- (3,1) -- (3,2) -- (2,1) -- cycle;
	\draw[draw=black] (2,0) -- (3,0) -- (4,1) -- (3,1) -- cycle;
	\draw[draw=black] (3,0) -- (4,0) -- (5,1) -- (4,1) -- cycle;
	
	\draw (0,1) -- (1,1) -- (1,2) -- (0,2) -- cycle;
	\draw (1,1) -- (2,2) -- (2,3) -- (1,2) -- cycle;
	\draw (1,1) -- (2,1) -- (3,2) -- (2,2) -- cycle;
	\draw (3,1) -- (4,1) -- (4,2) -- (3,2) -- cycle;
	\draw (4,1) -- (5,2) -- (5,3) -- (4,2) -- cycle;
	\draw (4,1) -- (5,1) -- (6,2) -- (5,2) -- cycle;
	
	\draw (0,2) -- (1,3) -- (1,4) -- (0,3) -- cycle;
	\draw (0,2) -- (1,2) -- (2,3) -- (1,3) -- cycle;
	\filldraw[draw=black, fill=blue!20] (2,2) -- (3,2) node[above left]{$\alpha_3$} -- (3,3) -- (2,3) -- cycle;
	\draw (3,2) -- (4,3) -- (4,4) -- (3,3) -- cycle; 
	\draw (3,2) -- (4,2) -- (5,3) -- (4,3) -- cycle;
	\filldraw[draw=black, fill=blue!20] (5,2) -- (6,2) node[above left]{$\alpha_6$} -- (6,3) -- (5,3) -- cycle;
	\draw (6,2) -- (7,3) -- (7,4) -- (6,3) -- cycle;
	
	\draw (0,3) -- (1,4) -- (1,5) -- (0,4) -- cycle;
	\draw (1,3) -- (2,4) -- (2,5) -- (1,4) -- cycle;
	\draw (1,3) -- (2,3) -- (3,4) -- (2,4) -- cycle;
	\draw (2,3) -- (3,3) -- (4,4) -- (3,4) -- cycle;
	\draw (4,3) -- (5,4) -- (5,5) -- (4,4) -- cycle;
	\draw (4,3) -- (5,3) -- (6,4) -- (5,4) -- cycle;
	\draw (5,3) -- (6,3) -- (7,4) -- (6,4) -- cycle;
	\draw (7,3) -- (8,4) -- (8,5) -- (7,4) -- cycle;
	
	\draw (1,4) -- (2,5) -- (2,6) -- (1,5) -- cycle;
	\filldraw[draw=black, fill=blue!20] (2,4) -- (3,4) node[above left]{$\alpha_3$} -- (3,5) -- (2,5) -- cycle;
	\draw (3,4) -- (4,5) -- (4,6) -- (3,5) -- cycle;
	\draw (3,4) -- (4,4) -- (5,5) -- (4,5) -- cycle;
	\draw (5,4) -- (6,5) -- (6,6) -- (5,5) -- cycle;
	\draw (5,4) -- (6,4) -- (7,5) -- (6,5) -- cycle;
	\draw (6,4) -- (7,4) -- (8,5) -- (7,5) -- cycle;
	
	\draw (2,5) -- (3,5) -- (3,6) -- (2,6) -- cycle;
	\draw (3,5) -- (4,6) -- (4,7) -- (3,6) -- cycle;
	\draw (4,5) -- (5,6) -- (5,7) -- (4,6) -- cycle;
	\draw (4,5) -- (5,5) -- (6,6) -- (5,6) -- cycle;
	\draw (6,5) -- (7,5) -- (7,6) -- (6,6) -- cycle;
	\draw (7,5) -- (8,5) -- (8,6) -- (7,6) -- cycle;
	
	\draw (2,6) -- (3,6) -- (4,7) -- (3,7) -- cycle;
	\draw (4,6) -- (5,7) -- (5,8) -- (4,7) -- cycle;
	\filldraw[draw=black, fill=blue!20] (5,6) -- (6,6) node[above left]{$\alpha_6$} -- (6,7) -- (5,7) -- cycle;
	\filldraw[draw=black, fill=blue!20] (6,6) -- (7,6) node[above left]{$\alpha_7$} -- (7,7) -- (6,7) -- cycle;
	\filldraw[draw=black, fill=blue!20] (7,6) -- (8,6) node[above left]{$\alpha_8$} -- (8,7) -- (7,7) -- cycle;
	
	\draw (3,7) -- (4,7) -- (5,8) -- (4,8) -- cycle;
	\draw (5,7) -- (6,7) -- (6,8) -- (5,8) -- cycle;
	\draw (6,7) -- (7,7) -- (7,8) -- (6,8) -- cycle;
	\draw (7,7) -- (8,7) -- (8,8) -- (7,8) -- cycle;
\end{tikzpicture}	
	\caption{\label{fig:ABC hexagon} An $ABC$-hexagon with side lengths $A = 4$, $B = 5$ and $C = 3$ (left) can be covered by the 
	three types of lozenges that are shown in the middle.
	The weighting \eqref{Wtile} depends on the coordinates
	$(x,y)$ of the square tiles. The right panel shows
	 a possible weighted tiling in which the shaded square tiles
 	 have the indicated weight $\alpha_j$ while all other
  tiles have weight $1$.}
\end{figure}
We follow \cite{charlier2020doubly, charlier2019periodic}.  
An $ABC$-hexagon has its vertices at the points $(0,0)$, $(B,0)$, $(B+C,C)$, $(B+C,A+C)$, $(C,A+C)$ and $(0,A)$ as in \cref{fig:ABC hexagon}.
The hexagon can then be covered with the three types of lozenges shown in \cref{fig:ABC hexagon} as well. This can be done in many
ways, and a particular example of a tiling is shown in the right panel
of \cref{fig:ABC hexagon} for the case of a regular
hexagon with $A=B=C=4$. 
The vertices of the lozenges all lie on the integer lattice
$\mathbb Z^2$.

\begin{figure}[t]
	\centering \vspace*{-1.5cm}
	\hspace*{-2cm}
	\includegraphics[scale=0.4]{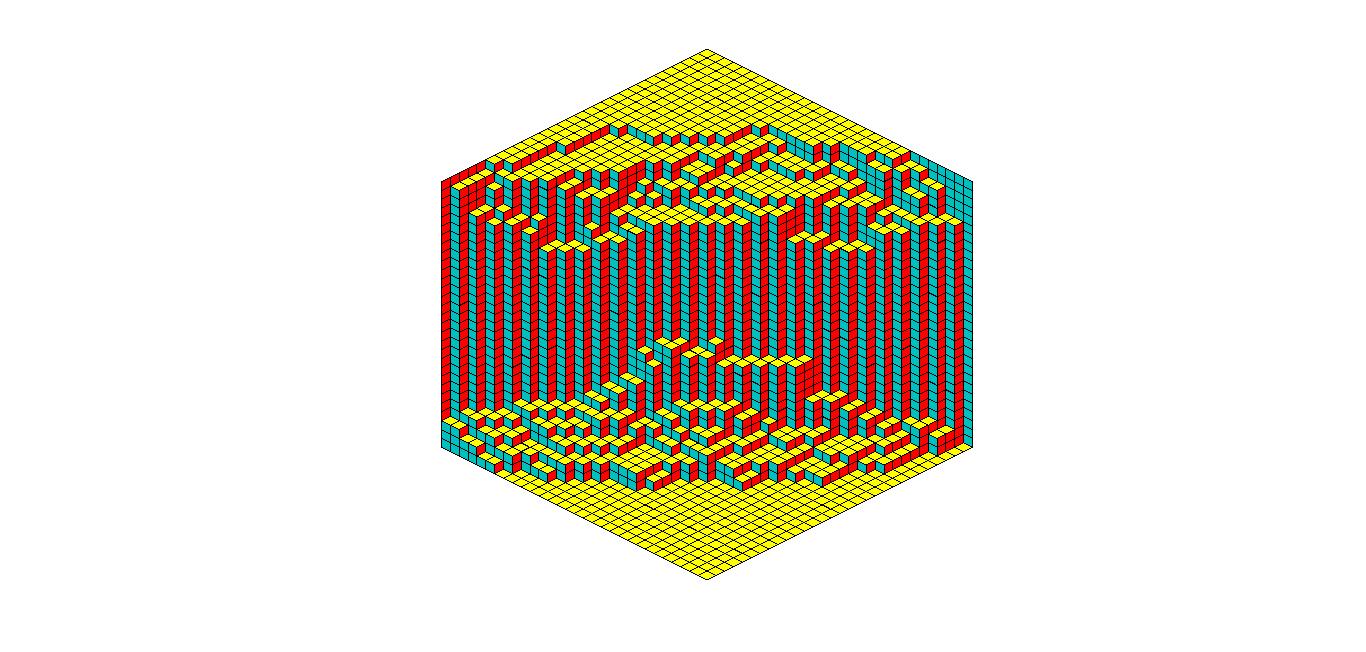} 
	\vspace*{-10mm}
	\caption{\label{fig:low alpha tiling}
		A typical lozenge tiling of a larger size hexagon 
		with periodic weighting with parameter $\alpha_j = 
		\alpha = 1/10$ for all $j$. 
		The regular hexagon in this picture is obtained from the one
		in \cref{fig:ABC hexagon} by a shear
		transformation $(x,y) \mapsto (x,y-x/2)$.
	The figure is due to Christophe Charlier.}
\end{figure}

In a weighted tiling model, certain weights are assigned
to the lozenges, depending on their shape and on their location
in the hexagon. In a periodic weighting this is done
in a periodic fashion with respect to at least one
direction. The simplest model is to introduce
two periodicity in one direction, say the vertical direction,
depending on the location of the square tiles. 
In this model we fix parameters $\alpha_1, \alpha_2, \ldots, \alpha_L > 0$ with $L=B+C$, and put
\begin{align} \label{Wtile}
\mathcal{W}\left(\begin{tikzpicture}[baseline={(0,0)},scale=0.5]
\draw[draw=black, very thick] (0,0) rectangle (1,1);
\filldraw (0,0) circle (2.5pt) node[below]{$(x,y)$};
\draw[white] (-1.30,-0.75) rectangle (1.25,1.15);
\end{tikzpicture} \right) = \begin{cases}
\alpha_j &\text{if $x = j-1$, and  $y$ is even}, \\
1 & \text{if $y$ is odd},
\end{cases}
\end{align}
while all other tiles have weight $1$.
A hexagon tiling $\mathcal T$ then has the weight
\begin{align*} \mathcal W(\mathcal T) = \prod_{T \in \mathcal T} 
\mathcal W(T)  = \prod_{j=1}^L \alpha_j^{N_j(\mathcal T)} \end{align*}
where $N_j(\mathcal T)$ is the number of square tiles in
the $j$th column of  $\mathcal T$ at an even height. 
This is a $1 \times 2$-periodic tiling of the hexagon, as the weights only depend on the height of the lozenge (that is, the vertical direction).
See the right panel of \cref{fig:ABC hexagon}
where the square tiles with weight $\alpha_j$ are highlighted.
We introduce a probability measure $\prob$ on the set of all tilings of a hexagon of fixed size by setting
\begin{align}
\prob(\mathcal{T}) = \frac{\mathcal{W}(\mathcal{T})}{\sum_{\mathcal{T}'} \mathcal{W}(\mathcal{T}')},
\end{align}
where the sum runs over all possible tilings $\mathcal T'$ 
of the hexagon.

There is a completely analogous weighting which is $2 \times 1$
periodic. Here the weight of the square tile depends on the
parity of the horizontal coordinate instead of the
vertical coordinate.  
The two weightings are equivalent and show the same phenomena
in the large size limit. 
The main phenomenon that was discovered is illustrated
in \cref{fig:low alpha tiling}.
In the large $n$
limit the pattern of tiles is fixed in certain regions near the
corners of the hexagon, where only one type of lozenge is present,
as well as in a region in the middle with two types of lozenges.
These regions are called the solid regions. The remaining
part of the hexagon is referred to as a liquid region.
All three types of lozenges are present in the liquid region,
and they do not appear in a regular pattern.

The solid region with two types of lozenges can connect two opposite sides of the 
hexagon as it is the case in \cref{fig:low alpha tiling}
or it can happen that it consists of two disjoint parts.
In the former case the liquid region consists of two disjoint
pieces, while in the latter case the liquid region 
is connected. There is a transition between the two cases
which depends on the parameter $\alpha > 0$. The critical
parameter is $\alpha = 1/9$ as shown in 
\cite{charlier2019periodic}.

\subsection{Systems of non-intersecting paths}
\label{subsec:tilings2}

A lozenge tiling of the hexagon can be alternatively viewed
as a non-intersecting path system. The paths correspond
to level lines for the height function for boxes 
stacked in a corner.
The non-intersecting paths are obtained by drawing 
diagonal lines on two of the three types of lozenges as
shown in \cref{fig:lines on boxes of same height}.
This gives us a bijection between tilings of an $ABC$-hexagon 
and non-intersecting path systems on a
directed graph $\mathcal{G} = (\mathbb Z^2, E)$
with starting points at $A$ consecutive 
points $(0,0), \ldots, (0,A-1)$ and ending points
at $(B+C,C), \ldots, (B+C,A+C-1)$. The edge set $E$
consists of directed edges  $((x,y), (x+1,y))$ and $((x,y), (x+1,y+1))$ 
with $(x,y) \in \mathbb Z^2$.
See \cref{fig:lines on boxes of same height}
for the non-intersecting path system corresponding
to the tiling of \cref{fig:ABC hexagon}.
The weights $\alpha_j$ on the square tiles in the
$1\times 2$ periodic setting correspond to edge weights on the horizontal
edges at even numbered heights. Note that we
applied a shift $y \mapsto y-1/2$ to the
vertical coordinate.

By then putting particles on the paths as also 
indicated in \cref{fig:lines on boxes of same height}, we obtain a multi-level particle system.
A hexagon of size $ABC$ then has $B+C+1$ levels $0,1, \ldots, B+C$
  with $A$ particles on each level.
The vertical positions of the particles on the $m$th 
level will be denoted by
\[
y_0^{(m)} < y_1^{(m)} < \cdots < y_{A-1}^{(m)}
\]
with fixed starting positions $y_j^{(0)} = j$ and ending positions
$y_j^{(L)} = C+j$ for $j=0, \ldots, A-1$, 
at levels $0$ and $B+C$, respectively. 

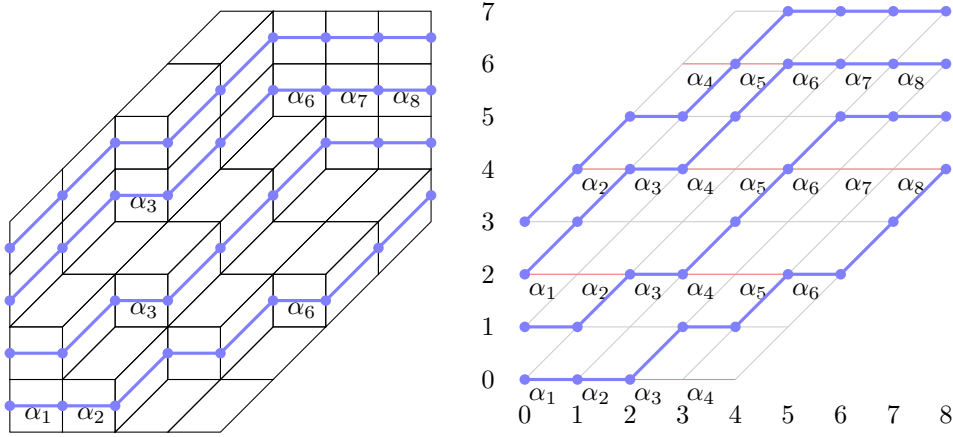
\begin{figure}[t]
	\centering
	\begin{tikzpicture}[scale=0.7]
	
	\draw[draw=black] (0,0) -- (1,0) node[above left]{$\alpha_1$} -- (1,1) -- (0,1) -- cycle;
	\draw[draw=black] (1,0) -- (2,0) node[above left]{$\alpha_2$} -- (2,1) -- (1,1) -- cycle;
	\draw[draw=black] (2,0) -- (3,1) -- (3,2) -- (2,1) -- cycle;
	\draw[draw=black] (2,0) -- (3,0) -- (4,1) -- (3,1) -- cycle;
	\draw[draw=black] (3,0) -- (4,0) -- (5,1) -- (4,1) -- cycle;
	
	\draw (0,1) -- (1,1) -- (1,2) -- (0,2) -- cycle;
	\draw (1,1) -- (2,2) -- (2,3) -- (1,2) -- cycle;
	\draw (1,1) -- (2,1) -- (3,2) -- (2,2) -- cycle;
	\draw (3,1) -- (4,1) -- (4,2) -- (3,2) -- cycle;
	\draw (4,1) -- (5,2) -- (5,3) -- (4,2) -- cycle;
	\draw (4,1) -- (5,1) -- (6,2) -- (5,2) -- cycle;
	
	\draw (0,2) -- (1,3) -- (1,4) -- (0,3) -- cycle;
	\draw (0,2) -- (1,2) -- (2,3) -- (1,3) -- cycle;
	\draw (2,2) -- (3,2) node[above left]{$\alpha_3$} -- (3,3) -- (2,3) -- cycle;
	\draw (3,2) -- (4,3) -- (4,4) -- (3,3) -- cycle; 
	\draw (3,2) -- (4,2) -- (5,3) -- (4,3) -- cycle;
	\draw (5,2) -- (6,2) node[above left]{$\alpha_6$} -- (6,3) -- (5,3) -- cycle;
	\draw (6,2) -- (7,3) -- (7,4) -- (6,3) -- cycle;
	
	\draw (0,3) -- (1,4) -- (1,5) -- (0,4) -- cycle;
	\draw (1,3) -- (2,4) -- (2,5) -- (1,4) -- cycle;
	\draw (1,3) -- (2,3) -- (3,4) -- (2,4) -- cycle;
	\draw (2,3) -- (3,3) -- (4,4) -- (3,4) -- cycle;
	\draw (4,3) -- (5,4) -- (5,5) -- (4,4) -- cycle;
	\draw (4,3) -- (5,3) -- (6,4) -- (5,4) -- cycle;
	\draw (5,3) -- (6,3) -- (7,4) -- (6,4) -- cycle;
	\draw (7,3) -- (8,4) -- (8,5) -- (7,4) -- cycle;
	
	\draw (1,4) -- (2,5) -- (2,6) -- (1,5) -- cycle;
	\draw (2,4) -- (3,4) node[above left]{$\alpha_3$} -- (3,5) -- (2,5) -- cycle;
	\draw (3,4) -- (4,5) -- (4,6) -- (3,5) -- cycle;
	\draw (3,4) -- (4,4) -- (5,5) -- (4,5) -- cycle;
	\draw (5,4) -- (6,5) -- (6,6) -- (5,5) -- cycle;
	\draw (5,4) -- (6,4) -- (7,5) -- (6,5) -- cycle;
	\draw (6,4) -- (7,4) -- (8,5) -- (7,5) -- cycle;
	
	\draw (2,5) -- (3,5) -- (3,6) -- (2,6) -- cycle;
	\draw (3,5) -- (4,6) -- (4,7) -- (3,6) -- cycle;
	\draw (4,5) -- (5,6) -- (5,7) -- (4,6) -- cycle;
	\draw (4,5) -- (5,5) -- (6,6) -- (5,6) -- cycle;
	\draw (6,5) -- (7,5) -- (7,6) -- (6,6) -- cycle;
	\draw (7,5) -- (8,5) -- (8,6) -- (7,6) -- cycle;
	
	\draw (2,6) -- (3,6) -- (4,7) -- (3,7) -- cycle;
	\draw (4,6) -- (5,7) -- (5,8) -- (4,7) -- cycle;
	\draw (5,6) -- (6,6) node[above left]{$\alpha_6$} -- (6,7) -- (5,7) -- cycle;
	\draw (6,6) -- (7,6) node[above left]{$\alpha_7$} -- (7,7) -- (6,7) -- cycle;
	\draw (7,6) -- (8,6) node[above left]{$\alpha_8$} -- (8,7) -- (7,7) -- cycle;
	
	\draw (3,7) -- (4,7) -- (5,8) -- (4,8) -- cycle;
	\draw (5,7) -- (6,7) -- (6,8) -- (5,8) -- cycle;
	\draw (6,7) -- (7,7) -- (7,8) -- (6,8) -- cycle;
	\draw (7,7) -- (8,7) -- (8,8) -- (7,8) -- cycle;
	
	\draw[very thick, color=blue!50] (0,0.5) -- (1,0.5) -- (2,0.5) -- (3,1.5) -- (4,1.5) -- (5,2.5) -- (6,2.5) -- (7,3.5) -- (8,4.5);
	
	\fill[color=blue!50] (0,0.5) circle (0.1);
	\fill[color=blue!50] (1,0.5) circle (0.1);
	\fill[color=blue!50] (2,0.5) circle (0.1);
	\fill[color=blue!50] (3,1.5) circle (0.1);
	\fill[color=blue!50] (4,1.5) circle (0.1);
	\fill[color=blue!50] (5,2.5) circle (0.1);
	\fill[color=blue!50] (6,2.5) circle (0.1);
	\fill[color=blue!50] (7,3.5) circle (0.1);
	\fill[color=blue!50] (8,4.5) circle (0.1);
	
	\draw[very thick, color=blue!50] (0,1.5) -- (1,1.5) -- (2,2.5) -- (3,2.5) -- (4,3.5) -- (5,4.5) -- (6,5.5) -- (7,5.5) -- (8,5.5);
	
	\fill[color=blue!50] (0,1.5) circle (0.1);
	\fill[color=blue!50] (1,1.5) circle (0.1);
	\fill[color=blue!50] (2,2.5) circle (0.1);
	\fill[color=blue!50] (3,2.5) circle (0.1);
	\fill[color=blue!50] (4,3.5) circle (0.1);
	\fill[color=blue!50] (5,4.5) circle (0.1);
	\fill[color=blue!50] (6,5.5) circle (0.1);
	\fill[color=blue!50] (7,5.5) circle (0.1);
	\fill[color=blue!50] (8,5.5) circle (0.1);
	
	\draw[very thick, color=blue!50] (0,2.5) -- (1,3.5) -- (2,4.5) -- (3,4.5) -- (4,5.5) -- (5,6.5) -- (6,6.5) -- (7,6.5) -- (8,6.5);
	
	\fill[color=blue!50] (0,2.5) circle (0.1);
	\fill[color=blue!50] (1,3.5) circle (0.1);
	\fill[color=blue!50] (2,4.5) circle (0.1);
	\fill[color=blue!50] (3,4.5) circle (0.1);
	\fill[color=blue!50] (4,5.5) circle (0.1);
	\fill[color=blue!50] (5,6.5) circle (0.1);
	\fill[color=blue!50] (6,6.5) circle (0.1);
	\fill[color=blue!50] (7,6.5) circle (0.1);
	\fill[color=blue!50] (8,6.5) circle (0.1);
	
	\draw[very thick, color=blue!50] (0,3.5) -- (1,4.5) -- (2,5.5) -- (3,5.5) -- (4,6.5) -- (5,7.5) -- (6,7.5) -- (7,7.5) -- (8,7.5);
	
	\fill[color=blue!50] (0,3.5) circle (0.1);
	\fill[color=blue!50] (1,4.5) circle (0.1);
	\fill[color=blue!50] (2,5.5) circle (0.1);
	\fill[color=blue!50] (3,5.5) circle (0.1);
	\fill[color=blue!50] (4,6.5) circle (0.1);
	\fill[color=blue!50] (5,7.5) circle (0.1);
	\fill[color=blue!50] (6,7.5) circle (0.1);
	\fill[color=blue!50] (7,7.5) circle (0.1);
	\fill[color=blue!50] (8,7.5) circle (0.1);
	\end{tikzpicture} \quad
	\begin{tikzpicture}[scale=0.7]
	\draw[color=black!20] (0,0.5) -- (4,0.5);
	\draw[color=black!20] (0,1.5) -- (5,1.5);
	\draw[color=black!20] (0,2.5) -- (6,2.5);
	\draw[color=black!20] (0,3.5) -- (7,3.5);
	\draw[color=black!20] (1,4.5) -- (8,4.5);
	\draw[color=black!20] (2,5.5) -- (8,5.5);
	\draw[color=black!20] (3,6.5) -- (8,6.5);
	\draw[color=black!20] (4,7.5) -- (8,7.5);
	
	\draw[color=black!20] (4,0.5) -- (8,4.5);
	\draw[color=black!20] (3,0.5) -- (8,5.5);
	\draw[color=black!20] (2,0.5) -- (8,6.5);
	\draw[color=black!20] (1,0.5) -- (8,7.5);
	
	\draw[color=black!20] (0,0.5) -- (7,7.5);
	\draw[color=black!20] (0,1.5) -- (6,7.5);
	\draw[color=black!20] (0,2.5) -- (5,7.5);
	\draw[color=black!20] (0,3.5) -- (4,7.5);

	\foreach \x in {0,1,2,3,4,5,6,7}
	{	\draw (-1,\x+0.5) node[right]{\x};	
		\draw (\x,-0.5) node[above]{\x};}
	\draw(8,-0.5) node[above]{8};
	
	\foreach \x in {0,2}
	\draw[color=red!50] (0,\x+0.5) -- (0.8,\x+0.5) node[below left, color=black]{$\alpha_1$} -- (1,\x+0.5); 
	\foreach \x in {0,2,4}
	\draw[color=red!50] (1,\x+0.5) -- (1.8,\x+0.5) node[below left, color=black]{$\alpha_2$} -- (2,\x+0.5) ;
	\foreach \x in {0,2,4}
	\draw[color=red!50] (2,\x+0.5) -- (2.8,\x+0.5) node[below left,color=black]{$\alpha_3$} -- (3,\x+0.5);
	\foreach \x in {0,2,4,6}
	\draw[color=red!50] (3,\x+0.5) -- (3.8,\x+0.5) node[below left,color=black]{$\alpha_4$} -- (4,\x+0.5);
	\foreach \x in {2,4,6}
	\draw[color=red!50] (4,\x+0.5) -- (4.8,\x+0.5) node[below left,color=black]{$\alpha_5$} -- (5,\x+0.5);
	\foreach \x in {2,4,6}
	\draw[color=red!50] (5,\x+0.5) -- (5.8,\x+0.5) node[below left,color=black]{$\alpha_6$} -- (6,\x+0.5);
	\foreach \x in {4,6}
	\draw[color=red!50] (6,\x+0.5) -- (6.8,\x+0.5) node[below left,color=black]{$\alpha_7$} -- (7,\x+0.5);
	\foreach \x in {4,6}
	\draw[color=red!50] (7,\x+0.5) -- (7.8,\x+0.5) node[below left,color=black]{$\alpha_8$} -- (8,\x+0.5);

	\draw[very thick, color=blue!50] (0,0.5) -- (1,0.5) -- (2,0.5) -- (3,1.5) -- (4,1.5) -- (5,2.5) -- (6,2.5) -- (7,3.5) -- (8,4.5);
	
	\fill[color=blue!50] (0,0.5) circle (0.1);
	\fill[color=blue!50] (1,0.5) circle (0.1);
	\fill[color=blue!50] (2,0.5) circle (0.1);
	\fill[color=blue!50] (3,1.5) circle (0.1);
	\fill[color=blue!50] (4,1.5) circle (0.1);
	\fill[color=blue!50] (5,2.5) circle (0.1);
	\fill[color=blue!50] (6,2.5) circle (0.1);
	\fill[color=blue!50] (7,3.5) circle (0.1);
	\fill[color=blue!50] (8,4.5) circle (0.1);
	
	\draw[very thick, color=blue!50] (0,1.5) -- (1,1.5) -- (2,2.5) -- (3,2.5) -- (4,3.5) -- (5,4.5) -- (6,5.5) -- (7,5.5) -- (8,5.5);
	
	\fill[color=blue!50] (0,1.5) circle (0.1);
	\fill[color=blue!50] (1,1.5) circle (0.1);
	\fill[color=blue!50] (2,2.5) circle (0.1);
	\fill[color=blue!50] (3,2.5) circle (0.1);
	\fill[color=blue!50] (4,3.5) circle (0.1);
	\fill[color=blue!50] (5,4.5) circle (0.1);
	\fill[color=blue!50] (6,5.5) circle (0.1);
	\fill[color=blue!50] (7,5.5) circle (0.1);
	\fill[color=blue!50] (8,5.5) circle (0.1);
	
	\draw[very thick, color=blue!50] (0,2.5) -- (1,3.5) -- (2,4.5) -- (3,4.5) -- (4,5.5) -- (5,6.5) -- (6,6.5) -- (7,6.5) -- (8,6.5);
	
	\fill[color=blue!50] (0,2.5) circle (0.1);
	\fill[color=blue!50] (1,3.5) circle (0.1);
	\fill[color=blue!50] (2,4.5) circle (0.1);
	\fill[color=blue!50] (3,4.5) circle (0.1);
	\fill[color=blue!50] (4,5.5) circle (0.1);
	\fill[color=blue!50] (5,6.5) circle (0.1);
	\fill[color=blue!50] (6,6.5) circle (0.1);
	\fill[color=blue!50] (7,6.5) circle (0.1);
	\fill[color=blue!50] (8,6.5) circle (0.1);
	
	\draw[very thick, color=blue!50] (0,3.5) -- (1,4.5) -- (2,5.5) -- (3,5.5) -- (4,6.5) -- (5,7.5) -- (6,7.5) -- (7,7.5) -- (8,7.5);
	
	\fill[color=blue!50] (0,3.5) circle (0.1);
	\fill[color=blue!50] (1,4.5) circle (0.1);
	\fill[color=blue!50] (2,5.5) circle (0.1);
	\fill[color=blue!50] (3,5.5) circle (0.1);
	\fill[color=blue!50] (4,6.5) circle (0.1);
	\fill[color=blue!50] (5,7.5) circle (0.1);
	\fill[color=blue!50] (6,7.5) circle (0.1);
	\fill[color=blue!50] (7,7.5) circle (0.1);
	\fill[color=blue!50] (8,7.5) circle (0.1);
	\end{tikzpicture}
	\caption{\label{fig:lines on boxes of same height} 
		The non-intersecting path system on a directed graph corresponding
		to the hexagon tiling of \cref{fig:ABC hexagon}.
		The $1\times 2$ periodic weighting corresponds to
		weights $\alpha$ on the horizontal edges at 
		even height. All other edges have weight $1$.}
\end{figure}

To any weighting of the edges of the directed graph
we associate transition matrices $T_m : \mathbb Z \times \mathbb Z \to \mathbb R$ for $m \in \mathbb Z$ where
$ T_m(x,y)$  is equal to the  weight on the edge from $(m,x)$ to $(m+1,y)$ if there is such an edge, and zero otherwise.

Then it is a general fact, which is a consequence
of the Lindstr\"om-Gessel-Viennot lemma 
\cite{gessel_viennot, lindstrom} that the multi-level particle system coming from an  $ABC$-hexagon  has the joint probability measure given by a product of determinants
\begin{multline}\label{eq:joint probability distribution}
P\left(\left(y_j^{(m)}\right)_{j=0,m=0}^{A-1,B+C}\right) 
\\ = \frac{1}{Z_{A,B,C}} \prod_{j=0}^{A-1}  \delta_j\left(y_j^{(0)}\right) \cdot \prod_{m=0}^{B+C-1} 
\det \left[T_m\left(y_j^{(m)},y_k^{(m+1)}\right)\right]_{j,k = 0}^{A-1} \cdot \prod_{j=0}^{A-1} \delta_{C+j} \left(y_j^{(B+C)}\right).
\end{multline}
The probability measure \eqref{eq:joint probability distribution}
is determinantal by the Eynard-Mehta theorem \cite{eynard_mehta} (see also \cite[Theorem 4.3]{duits2017periodic})
with a correlation kernel that has an explicit  double
sum formula.

\subsection{Periodic transition matrices} \label{subsec:tilings3}
One of the contributions of \cite{duits2017periodic} is
to rewrite the correlation kernel of the determinantal
point process as a double contour integral in case all
transition matrices $T_m$ are periodic, in the sense that
\[ T_m(x+p,y+p) = T_m(x,y), \qquad (x,y) \in \mathbb Z^2 \]
for some integer $p \geq 1$. 
The MVOP appear in this formula and they have size $p \times p$
if the transition matrices are $p$-periodic.  
The orthogonality weights come from the symbols of the
transition matrices.

In the situation of the  two periodic weighting
with parameters $\alpha_1, \ldots, \alpha_{B+C}$, the transition 
matrices $T_m : \Z \times \Z \to \R$ 
are $2$-periodic  
\begin{align}
T_m(x,y) = \begin{cases}
\alpha_j, &\text{if $m=j-1$, and $y = x$ is even}, \\
1, &\text{if $m=j-1$, and $y = x$ is odd or if $y = x+1$}, \\
0, &\text{otherwise}.
\end{cases}
\end{align}
The symbol of $T_m$ is given by
\begin{align}\label{eq:matrix symbol}
W_{\alpha}(z) = \begin{pmatrix}
\alpha & 1 \\
z & 1
\end{pmatrix}, \qquad \text{with $\alpha = \alpha_j$, $j=m+1$}
\end{align}
and these will enter into the matrix-valued weight function.

The formalism of \cite{duits2017periodic} 
then assumes that the sides $A$ and $C$ of the hexagon are
even, say $A=2k$, $C=2c$.
The matrix-valued orthogonality weight then is
\begin{equation} \label{eq:weightABC} 
	\frac{W_{\alpha_1}(z) W_{\alpha_2}(z) \cdots W_{\alpha_{B+C}}(z)}{z^{k+c}}, 
	\end{equation}
and Lemma 4.8 of \cite{duits2017periodic} then states that
a unique monic matrix-valued polynomial $P_k$ of degree $k$
exists that is orthogonal with respect to the above
weight matrix on a closed contour $\gamma$ going
around $0$ once in the counterclockwise direction. 
This leads to the proof of the existence of MVOP.

\subsection{Proof of \cref{prop:exist}} \label{subsec:tilings4}

\begin{proof}
Suppose an integer $k \geq 0$ is given, as well
as integers $K \geq k$ and  $L \geq 2(K-k)$. 
We take 
\[ A = 2k, \quad  c = K-k, \quad C=2c, \quad \text{and} \quad B = L-C. \]
Then $A,B,C$ are non-negative integers with $A$ and $C$ even.
With these parameters the matrix weight
\eqref{eq:weightABC} is equal to \eqref{eq:weightW}
and therefore the MVOP $P_k$ of degree $k$ uniquely exists,
due to \cite[Lemma 4.8]{duits2017periodic}.
\end{proof}

\section{Proof of \cref{thm:theorem1}}\label{sec:proof of main result}

\begin{proof}
(a)
Let 
\begin{equation} \label{eq:Q2k} 
	Q_{2k}(\zeta) = P_k(\zeta^2-\beta^2)
	\begin{pmatrix} 1 & 1 \\ \zeta + \beta & - \zeta+ \beta
	\end{pmatrix}. 
	\end{equation}
Then the first column of $Q_{2k}$ has the entries 
\begin{equation} \label{eq:q2k}
\begin{aligned} 
	q_{2k}(\zeta)  := \left( Q_{2k}\right)_{11}(\zeta) & =  \left(P_k\right)_{11}(\zeta^2 - \beta^2)
		+ \left(P_k\right)_{12}(\zeta^2 - \beta^2)
			(\zeta+\beta), \\ 
	q_{2k+1}(\zeta) := \left(Q_{2k}\right)_{21}(\zeta) & =  \left(P_k\right)_{21}(\zeta^2 - \beta^2)
	+ \left(P_k\right)_{22}(\zeta^2 - \beta^2)
	(\zeta+\beta),
\end{aligned} \end{equation} 
and these are easily seen to be monic polynomials
of respective degrees $2k$ and $2k+1$, since $P_k$ is 
a monic matrix valued polynomial of degree $k$.
From \eqref{eq:Q2k} it is also immediate that the second
column contains the entries $q_{2k}(-\zeta)$ and $q_{2k+1}(-\zeta)$ with the same polynomials \eqref{eq:q2k}. 
Thus \eqref{eq:Pkformula1} holds
and it remains to check the orthogonality \eqref{eq:scalarOP1}.
	
The matrix orthogonality of $P_k$ with weight \eqref{eq:matrixW1} yields
\begin{align} \label{eq:Pkintegral}
	\frac{1}{2 \pi i} \oint_{\gamma} P_{k}(z) W_\alpha(z)^{L} 
	z^{j-K} d z = 0_2, \quad j = 0, 1, \ldots, k-1,
	\end{align}
where we choose for $\gamma$ the circle around $-\beta^2$ of radius $R^2 > \beta^2$. 
Recall the spectral decomposition $W_{\alpha} = E \Lambda E^{-1}$ from \eqref{eq:Walpha} with
\begin{align} \label{eq:LambdaE}
	\Lambda(\zeta) & = \begin{pmatrix}
	\zeta + \alpha +\beta & 0 \\ 0 & -\zeta + \alpha + \beta \end{pmatrix}, \qquad 
	 E(\zeta)  = \begin{pmatrix} 1 & 1 \\ \zeta +\beta & 
	-\zeta + \beta \end{pmatrix}, 
	\end{align}
and $\zeta = (z+\beta^2)^{1/2}$, $\Re \zeta > 0$,
see \eqref{eq:Lambda} and \eqref{eq:Ealpha}. Hence by \eqref{eq:Q2k}, \eqref{eq:Walpha}, \eqref{eq:LambdaE},
\begin{align*} 
	P_k(z) W_{\alpha}(z)^L  & =
	Q_{2k}(\zeta)  \Lambda(\zeta)^L E^{-1}(\zeta) = 
	Q_{2k}(\zeta) \Lambda(\zeta)^L \frac{1}{2\zeta}
		\begin{pmatrix} \zeta & 1 \\
		\zeta & -1 \end{pmatrix} 
		\begin{pmatrix} 1 & 0 \\  -\beta & 1 
		\end{pmatrix}.
	\end{align*}
Performing the change of variable $z = \zeta^2 - \beta^2$ 
with $d z = 2 \zeta d \zeta$ in \eqref{eq:Pkintegral}, 
we obtain
\begin{align} \label{eq:Q2kintegral}
	\frac{1}{2 \pi i} \int_{\gamma_R^+} Q_{2k}(\zeta) 
	\Lambda(\zeta)^L  
		\begin{pmatrix} \zeta & 1 \\
	\zeta & -1 \end{pmatrix} 
	\begin{pmatrix} 1 & 0 \\ -\beta & 1 
	\end{pmatrix}
		(\zeta^2-\beta^2)^{j-K} 
	 d \zeta = 0_2, \quad j = 0,1,\ldots,k-1,
	\end{align}
	where $\gamma_R^+$ denotes the semicircle $|\zeta| = R$,
	$\Re \zeta > 0$.  
The constant matrix $\begin{pmatrix} 1 & 0 \\
-\beta & 1 \end{pmatrix}$ can be dropped from  \eqref{eq:Q2kintegral} as it disappears if we
multiply by its inverse. Then the matrix identity \eqref{eq:Q2kintegral} 
results in four integrals that are equal to zero, and
each integral contains the scalar weight $w_{\alpha,K,L}$
from \eqref{eq:scalarwKL}.
Using \eqref{eq:Q2k}, \eqref{eq:Pkformula1}, \eqref{eq:LambdaE},
and \eqref{eq:scalarwKL}, we obtain the following integrals containing either $q_{2k}$ or $q_{2k+1}$,
\begin{align}\label{eq:equations obtained from matrix orthogonality}
\begin{aligned}
	\frac{1}{2 \pi i} \int_{\gamma_R^+} \left[q_{2k+\epsilon}(\zeta) w_{\alpha,K,L}(\zeta) + q_{2k+\epsilon}(-\zeta) w_{\alpha,K,L}(-\zeta) \right] \,
	\zeta (\zeta^2 - \beta^2)^j d \zeta&= 0, \\
	\frac{1}{2 \pi i} \int_{\gamma_R^+} \left[q_{2k+\epsilon}(\zeta) w_{\alpha,K,L}(\zeta) - q_{2k+\epsilon}(-\zeta) w_{\alpha,K,L}(-\zeta)\right] \, (\zeta^2 - \beta^2)^j d \zeta &= 0, 
\end{aligned}
\end{align}
for $\epsilon = 0,1$, and $j=0, 1, \ldots, k-1$. 
	
The integrals in \eqref{eq:equations obtained from matrix orthogonality} are turned into integrals over the full
circle $\gamma_R$ by writing them as a sum of two integrals
and changing variables $\zeta \to -\zeta$ in the second one.
We obtain for $j=0,1, \ldots, k-1$, and $\epsilon = 0,1$,
\begin{equation}	
	\label{eq:orthogonality q11}
	\begin{aligned}
\frac{1}{2 \pi i} \oint_{\gamma_R} q_{2k+\epsilon}(\zeta) w_{\alpha,K,L}(\zeta)  \zeta (\zeta^2 - \beta^2)^j d \zeta & = 0, \\
\frac{1}{2 \pi i} \oint_{\gamma_R} q_{2k+\epsilon}(\zeta) w_{\alpha,K,L}(\zeta)  (\zeta^2 - \beta^2)^j d \zeta & = 0.
	\end{aligned}
\end{equation}
Observe the extra factor $\zeta$ in the first integral of \eqref{eq:orthogonality q11}.
The $2k$ polynomials $\zeta \mapsto \zeta(\zeta^2 - \beta^2)^j$, $\zeta \mapsto (\zeta^2 - \beta^2)^j$, $j=0, \ldots, k-1$, are linearly independent, since we have a unique
polynomial in this set for each degree from $0$ up to $2k-1$.
Thus they are a basis of the vector space of all polynomials
of degree $\leq 2k-1$. 
We then obtain \eqref{eq:scalarOP1} by taking suitable linear combinations of 
the identities in \eqref{eq:orthogonality q11}, which completes
the proof of part (a).

\medskip 
(b)
Suppose $q_{2k}$ and $q_{2k+1}$ are two monic polynomials 
of degrees $2k$ and $2k+1$, respectively that satisfy 
\eqref{eq:scalarOP1} and define $Q_{2k}$ as
\[ Q_{2k}(\zeta) = \begin{pmatrix} q_{2k}(\zeta) & q_{2k}(-\zeta) \\
	q_{2k+1}(\zeta) & q_{2k+1}(-\zeta) \end{pmatrix}.	\]
	
Then, for any $c \in \C$, the matrix valued function $P_k$ 
as defined in \eqref{eq:scalarOP2} satisfies the orthogonality 
\eqref{eq:MVOP1} with weight matrix \eqref{eq:matrixW1},
simply by reversing the arguments from part (a).
From the second line of \eqref{eq:scalarOP2} it is clear that
$P_k$ is a polynomial in the variable $z$, see also the remark
following the statement of Theorem~\ref{thm:theorem1}.
The diagonal entries of $P_k$ are monic polynomials of degree $k$,
and the $(1,2)$-entry is a polynomial of degree $\leq k-1$.
The $(2,1)$-entry is a polynomial of degree $\leq k$.

Suppose 
\[ q_{2k+1}(\zeta) = \zeta^{2k+1} + a \zeta^{2k} + O(\zeta^{2k-1}) \]
as $\zeta \to \infty$. Then
\[\frac{q_{2k+1}(\zeta) + q_{2k+1}(-\zeta)}{2} 
	= a \zeta^{2k} + O (\zeta^{2k-2}) 
	\]
and it follows from \eqref{eq:scalarOP2} that 
\[ \left(P_k \right)_{21}(z) = (c + a - \beta ) z^k + 
	O(z^{k-1}) \quad \text{as } z \to \infty. \]
The choice for $c$ in \eqref{eq:constantc} then guarantees that $c+a - \beta = 0$
and therefore the $(2,1)$ entry of $P_k$ has degree $\leq k-1$.
Thus $P_k$ defined by \eqref{eq:scalarOP2} is a monic matrix valued
polynomial with the required matrix orthogonality. 
Part (b) then follows because of the uniqueness of the MVOP,
see also \cref{prop:exist}.
\end{proof}

\section{Proofs of \cref{prop:zerosqnn}
and \cref{thm:zero distribution determinant}}
\label{sec:strong asymptotics}

We prove part (a) of \cref{prop:zerosqnn} via the 
Deift-Zhou steepest
descent analysis of the Riemann-Hilbert problem (or RH problem) for scalar orthogonal polynomials.
It is well-known that the monic scalar orthogonal polynomials can be expressed in terms of a RH problem \cite{fokas_its_kitaev}.
The Deift-Zhou steepest descent analysis was first performed in \cite{deift_zhou_steepest_descent} and its application to
orthogonal polynomials has been well-developed by now. 
We refer to \cite{deift_orthogonal_1999, deift_uniform_1999} or \cite[Section 2.4]{bleher_liechty} for more background information.

In our case, the method works smoothly for $q_{n,n}$ but we need a little 
twist to make it work for $q_{n+k,n}$ with $k \neq 0$. 
This is the reason why we give a detailed account here.

\subsection{The Riemann-Hilbert problem for $q_{n+k,n}$}

The critical trajectory $\Sigma_0$ is analytically continued to 
a closed contour $\gamma_0$ around $0$ as in \cref{fig:helpful picture for definition}. The additional part is an orthogonal trajectory of
the quadratic differential $Q_{\alpha} dz^2$. Then $\gamma_0$
surrounds the points $\pm \beta$, see also 
\cite[section 4]{charlier2019periodic}.
Thus $\gamma_R$ can be deformed to $\gamma_0$ in \eqref{eq:scalarqkn} without crossing the poles,
and $q_{n+k,n}$ is also characterized by the orthogonality
\[ \frac{1}{2\pi i} \oint_{\gamma_0} q_{n+k,n}(z)
z^j  \frac{(z+\alpha+\beta)^{2n}}{(z^2-\beta^2)^n} dz = 0,
\qquad j =0,1, \ldots, n+k-1. \]
The closed contour $\gamma_0$ is positively (counterclockwise)
oriented. As usual in RH problems the $+$ side
is on the left and the $-$ side is on the right when traversing
the contour according to its orientation.

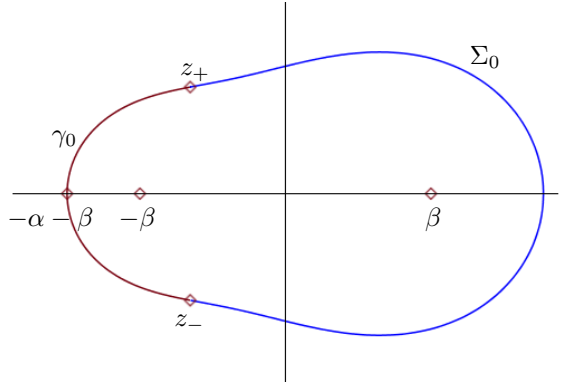
\begin{figure}
	\centering
\begin{overpic}[scale=0.7]{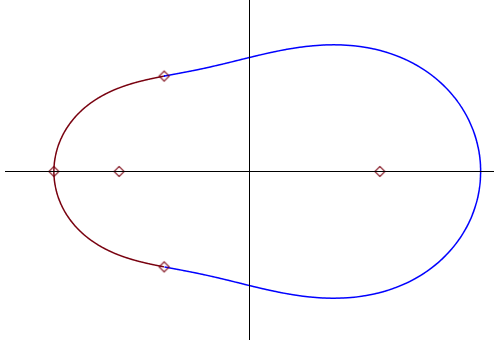}
	\put(83,57){$\Sigma_0$}
	\put(8,43){$\gamma_0$}
	\put(31,55){$z_+$}
	\put(30,10){$z_-$}
	\put(0,28){$-\alpha-\beta$}
	\put(20,28){$-\beta$}
	\put(75,28){$\beta$}	
	\end{overpic}
	\caption{\label{fig:helpful picture for definition} 
		$\Sigma_0$ (in blue) has analytic continuation to a closed
		contour $\gamma_0$ surrounding $\pm \beta$, 
		where $\gamma_0 \setminus \Sigma_0$
		consists of orthogonal trajectories connecting
		$z_{\pm}$ with $-\alpha-\beta$.}
\end{figure}

\begin{rhp} \label{rhp:Y}
 $Y$ is a $2 \times 2$ matrix-valued function that
 satisfies 
	\begin{itemize}
		\item $Y : \C \setminus \gamma_0 \to \C^{2 \times 2}$ is analytic,
		\item $Y$ has a jump on $\gamma_0$ given by
		\begin{align}
			Y_+(z) = Y_-(z) \begin{pmatrix}
				1 & w_\alpha(z)^{n} \\ 0 & 1
			\end{pmatrix} \quad \text{ for } z \in \gamma_0,
		\end{align}
		where $Y_{\pm}$ denote the limiting values of $Y$ on the $\pm$ sides of $\gamma_0$, and 
		\[ w_{\alpha}(z) = \frac{(z+\alpha+\beta)^{2}}{(z^2-\beta^2)}. \]
		\item $Y(z) = (I_2 + \bigO(z^{-1})) \begin{pmatrix}
			z^{n+k} & 0 \\ 0 & z^{-n-k}
		\end{pmatrix}$ as $z \to \infty$.
	\end{itemize}	
\end{rhp}

Fokas, Its and Kitaev first established in \cite{fokas_its_kitaev} that
\cref{rhp:Y} has a solution in terms of the scalar orthogonal polynomials $q_{n+k,n}$. The solution to the RH problem exists
if and only if $q_{n+k,n}$ uniquely exists, and in that case 
\begin{align}
	Y(z) = \begin{pmatrix}
		q_{n+k,n}(z) & \ds \frac{1}{2 \pi i} \oint_{\gamma_0} \frac{q_{n+k,n}(s) w_\alpha(s)^{n}}{s - z} \d s \\
		\widehat{q}_{n+k-1,n}(z) & \ds  \frac{1}{2 \pi i} \oint_{\gamma_0} \frac{\widehat{q}_{n+k-1,n}(s) w_\alpha(s)^{n}}{s - z} \d s
	\end{pmatrix}
\end{align}
for some polynomial $\widehat{q}_{n+k-1,n}$ of degree
$\leq n+k-1$. If $\widehat{q}_{n+k-1,n}$ has exact
degree $n+k-1$, then it is a multiple of the
monic orthogonal polynomial of degree $n+k-1$,
and $q_{n+k-1,n}$ uniquely exists as well. However, it is 
possible that the degree of $\widehat{q}_{n+k-1,n}$ is less than $n+k-1$. 

A consequence of the steepest descent analysis will be that, 
for a fixed $k \in \mathbb Z$, the \cref{rhp:Y} is solvable for $n$ large enough. Hence the polynomial $q_{n+k,n}$ exists for $n$ large enough.
\subsection{Steepest descent analysis}

\subsubsection{First transformation $Y \mapsto T$}

Note that there is a discrepancy between the power of the weight function in the jump of $Y$ and the degree of $z^{\pm (n+k)}$ in the asymptotics of $Y(z)$ as $z \to \infty$.
The transformation $Y \mapsto T$ will not depend on $k$,
and it leads to a RH problem that is not normalized
at infinity, as we will see.  We use the $g$-function
\eqref{eq:gzdef}, the function
\begin{align} \label{eq:Valpha} V_{\alpha}(z) = - \log w_{\alpha}(z)
	= \log(z^2-\beta^2) - 2\log(z+\alpha+\beta),
\end{align} 
as well as the following lemma.

\begin{lemma} Let $\frac{1}{9} < \alpha \leq 1$, and let $
	\Sigma_0$ and $\gamma_0$ be as in \cref{fig:helpful picture for definition}. Then there is
	a constant $\ell$ such that
	\begin{equation} \label{eq:gjumps}
		g_+(z) + g_-(z) - V_{\alpha}(z) \begin{cases}
		= - \ell, &  z \in \Sigma_0, \\
		< - \ell, &  z \in \gamma_0 \setminus \Sigma_0.
	\end{cases} \end{equation}
	In addition
	\begin{equation}
		g_+(z) - g_-(z) = 2 \phi_{+}(z), \qquad z \in \Sigma_0,
	\end{equation}
	where $\phi$ is defined by 
	\begin{equation} \label{eq:phi} 
	\phi(z) = 	\int_{z_+}^z Q_{\alpha}(s)^{1/2} ds 
	\end{equation}
	and $\phi$ satisfies
	\begin{equation} \label{eq:phiidentity} 
		\phi(z) =  g(z) - \frac{V_{\alpha}(z)}{2} + 
		\frac{\ell}{2}, \qquad z \in \mathbb C \setminus 
		(\Sigma_0 \cup (-\infty, x^*]) \end{equation}
	where $x^*$ is the point of intersection of $\Sigma_0$
	with the positive real line.
\end{lemma} 
\begin{proof}
	This is essentially the same as the proof of  
	\cite[Proposition 4.4]{charlier2019periodic}.
\end{proof}
By taking the derivative in \eqref{eq:phiidentity} and using
\eqref{eq:phi} we get that  
\begin{equation} \label{eq:gprimeidentity} 
	2 g'(z) = V_{\alpha}'(z) + 2 Q_{\alpha}(z)^{1/2},
	\qquad z \in \mathbb C \setminus \Sigma_0.
	\end{equation} 
Also note that $\mu_0$ is a critical measure in the sense of Mart\'{i}nez-Finkelshtein 
and Rakhmanov, see \cite{mf_rakhmanov_critical_2011}
and especially Section~5 therein.

We then define
\begin{align} \label{eq:Tdef}
	T(z) = e^{\frac{n \ell}{2} \sigma_3} Y(z) e^{-n g(z) \sigma_3} e^{-\frac{n \ell}{2} \sigma_3 },
	\qquad \quad \sigma_3 = \begin{pmatrix}
		1 & 0 \\ 0 & -1 \end{pmatrix}.
\end{align}
Then $T$ satisfies the following Riemann-Hilbert problem.

\begin{rhp} \label{rhp:T} The matrix valued function $T$ satisfies the following:
	\begin{itemize}
		\item $T : \C \setminus \gamma_0 \to \C^{2 \times 2}$ is analytic,
		\item $T$ has boundary values on $\gamma_0$ that satisfy
		\begin{align}
			T_+(z) &= T_-(z) \begin{pmatrix}
				e^{- 2n \phi_+(z)} & 1 \\ 0 & e^{- 2n \phi_-(z)}
			\end{pmatrix}, &&\text{for $z \in \Sigma_0 \subset \gamma_0$,} \\
			T_+(z) &= T_-(z) \begin{pmatrix}
				1 & e^{2n \phi(z)} \\ 0 & 1
			\end{pmatrix}, &&\text{for $z \in \gamma_0 \setminus \Sigma_0$,}
		\end{align}
		\item $T(z) = \left(I_2 + \bigO(z^{-1})\right) \begin{pmatrix}
			z^k & 0 \\ 0 & z^{-k}
		\end{pmatrix}$ as $z \to \infty$.
	\end{itemize}
\end{rhp}
The jump matrix on $\gamma_0 \setminus \Sigma_0$ tends
to the identity matrix as $n \to \infty$ because
$\phi(z) < 0$ on $\gamma_0 \setminus \Sigma_0$
due to \eqref{eq:phiidentity} and the strict inequality in \eqref{eq:gjumps}.

\subsubsection{Second transformation $T \mapsto S$: opening of lenses}\label{sec:opening of lenses}

We open up lenses around $\Sigma_0$ such that $\Re \phi > 0$ 
on the lips $\gamma_+$ and $\gamma_-$ of the lenses,
except at the endpoints $z_+, z_-$, and we define
\begin{align} \label{eq:Sdef}
	S(z) = T(z) \times \begin{cases}
		\begin{pmatrix}
			1 & 0 \\ -e^{-2n \phi(z)} & 1
		\end{pmatrix} &\text{for $z$ between $\Sigma_0$ and $\gamma_+$}, \\
		\begin{pmatrix}
			1 & 0 \\ e^{-2n \phi(z)} & 1
		\end{pmatrix} &\text{for $z$ between $\Sigma_0$ and $\gamma_-$}, \\
		I_2 &\text{elsewhere}.
	\end{cases}
\end{align}
Then $S$ satisfies the following Riemann-Hilbert problem:
\begin{rhp} \label{rhp:S}
	The function $S$ satisfies:
	\begin{itemize}
		\item $S : \C \setminus (\gamma \cup \gamma_+ \cup \gamma_-) \to \C^{2 \times 2}$ is analytic,
		\item $S$ has boundary values on $\gamma$, $\gamma_+$ and $\gamma_-$ that satisfy
		\begin{align}
			S_+(z) &= S_-(z) \begin{pmatrix}
				1 & 0 \\ e^{-2n \phi(z)} & 1
			\end{pmatrix} &&\text{for $z \in \gamma_+ \cup \gamma_-$}, \\
			S_+(z) &= S_-(z) \begin{pmatrix}
				0 & 1 \\ -1 & 0
			\end{pmatrix} &&\text{for $z \in \Sigma_0$}, \\
			S_+(z) &= S_-(z) \begin{pmatrix}
				1 & e^{2n \phi(z)} \\ 0 & 1
			\end{pmatrix} &&\text{for $z \in \gamma_0 \setminus \Sigma_0$},
		\end{align}
		\item $S(z) = (I_2 + \bigO(z^{-1})) \begin{pmatrix}
			z^k & 0 \\ 0 & z^{-k}
		\end{pmatrix}$ as $z \to \infty$.
	\end{itemize}
\end{rhp}
The jump matrices on $\gamma_{\pm}$ and on $\gamma_0 \setminus \Sigma_0$ tend to the identity matrix as $n \to \infty$.

\subsubsection{Model Riemann-Hilbert problem}

Ignoring the jumps that are exponentially small, we look for
a solution $N$ to the following model Riemann-Hilbert problem:
\begin{rhp} \label{rhp:N} The function $N$ satisfies
	\begin{itemize}
		\item $N : \C \setminus \Sigma_0 \to \C^{2 \times 2}$ is analytic,
		\item $N$ has boundary values on $\Sigma_0$ that satisfy
		\begin{align}
			N_+(z) = N_-(z) \begin{pmatrix}
				0 & 1 \\ -1 & 0
			\end{pmatrix}  \quad \text{ for } z \in \Sigma_0, 
		\end{align}
		\item $N(z) = \left(I_2 + \bigO(z^{-1})\right) \begin{pmatrix}
			z^k & 0 \\ 0 & z^{-k}
		\end{pmatrix}$ as $z \to \infty$.
	\end{itemize}
\end{rhp}
The \cref{rhp:N} depends on the integer $k \in \mathbb Z$
as it appears in the asymptotic condition.

The solution $N_0$ for the case $k=0$ is well-known, namely
\begin{align}\label{eq:standard solution model RHP}
	N_0(z) = \begin{pmatrix}
	\ds	\frac12 (a(z) + a(z)^{-1}) & \ds \frac1{2i} (a(z) - a(z)^{-1}) \\[5pt]
	\ds	-\frac1{2i} (a(z) - a(z)^{-1}) & \ds \frac12 (a(z) + a(z)^{-1})
	\end{pmatrix},
\end{align}
where  $a(z)$ is given by \eqref{eq:az}. Note that 
$(N_0)_{11}(z) = A_0(z)$ where $A_0(z)$ is given by \eqref{eq:A0z}.

\begin{lemma}
 The solution of \cref{rhp:N} takes the form 
\begin{align}\label{eq:presumed solution of model RHP}
	N(z) = \begin{pmatrix} c^k & 0 \\ 0 & c^{-k}  \end{pmatrix} N_0(z) \begin{pmatrix} \left(\frac{\psi(z)}{c}\right)^k & 0 \\ 0 & \left(\frac{\psi(z)}{c} \right)^{-k} \end{pmatrix}
\end{align}
where $c = \frac{z_+-z_-}{4}$ and $\psi(z)$ is given by \eqref{eq:psiz}.
\end{lemma}
\begin{proof}
	The function $\psi$ is analytic and non-zero on
	$\mathbb C \setminus \Sigma_0$ with boundary
	values satisfying $\psi_+ \psi_- = c^2$ on $\Sigma_0$, as can be 
	easily obtained from \eqref{eq:psiz}. These properties
	guarantee that $N$ is well-defined and analytic outside
	$\Sigma_0$ with the same jump matrix on $\Sigma_0$
	as $N_0$ has.    Also 
	$\psi(z) = z + \bigO(1)$ as $z \to \infty$
	and this gives the asymptotic condition from the
	\cref{rhp:N}.
\end{proof}

\subsubsection{Local parametrices and final transformation $S \mapsto R$}

In small disks $D_{+}$ and $D_{-}$ around the endpoints $z_+$ and $z_-$ of $\Sigma_0$, we can build local parametrices using Airy functions, see for example
\cite{deift_orthogonal_1999, deift_uniform_1999} or  \cite[Section 2.4.6]{bleher_liechty}.
We will call these $P^{(+)}$ and $P^{(-)}$ respectively.
The local parametrices do not play a role in the strong asymptotics of $q_{n+k,n}$ away from the contour $\Sigma_0$.

We next define the function $R$ by
\begin{align} \label{eq:Rdef}
	R(z) = S(z) \times \begin{cases}
		N(z)^{-1} &\text{ for } z \in \C \setminus (D_{+} \cup D_-), \\
		P^{(+)}(z)^{-1} &\text{ for } z \in D_+, \\
		P^{(-)}(z)^{-1} &\text{ for } z \in D_-.
	\end{cases}
\end{align}
Then the following lemma holds, cf.\  \cite[Section 2.5]{bleher_liechty} or \cite[Lemma 8.3]{kuijlaars_riemann_hilbert_2004}.

\begin{lemma}\label{lem:asymptotics final transformation}
	$R$ is the solution of a small-norm RH problem, which
	has a solution for $n$ large enough. In addition,
	there exists a constant $C > 0$  such that for 
	all $z \in \mathbb C$, 
	\[
	\| R(z) - I_2 \| \leq \frac{C}{n (1+|z|)},
	\]
	where $\| \cdot \|$ denotes any matrix norm. 
\end{lemma}

\subsubsection{Proof of \cref{prop:zerosqnn}(a)}
Since $R$ exists for $n$ large enough, we find by tracing
back the transformations $Y \mapsto T \mapsto S \mapsto R$ 
that the \cref{rhp:Y} is solvable
for $n$ large enough. This implies that $q_{n+k,n}$
exists for $n$ large enough.

Let $K$ be a compact subset of $\overline{\C} \setminus \Sigma_0$.
By taking the lense around $\Sigma_0$ and the disks 
around $z_{\pm}$ small enough such that 
$K$ lies outside the region bounded by the lenses and disks, 
we find from \eqref{eq:Tdef}, \eqref{eq:Sdef}, \eqref{eq:Rdef} that
\[ Y(z) = e^{-\frac{n \ell}{2} \sigma_3} R(z) N(z) e^{\frac{n \ell}{2} \sigma_3} e^{n g(z) \sigma_3}, \quad z \in K.
\]
Because $q_{n+k,n}(z) = Y_{11}(z)$ is the $(1,1)$-entry of $Y$
we then obtain  
\begin{align*}
	q_{n+k,n}(z) & = \begin{pmatrix}
	1 & 0
	\end{pmatrix} Y(z) \begin{pmatrix}
	1 \\ 0
	\end{pmatrix} \\ 
	& = 	
	e^{n g(z)} \begin{pmatrix}
		1 & 0
	\end{pmatrix} R(z) N(z) \begin{pmatrix}
		1 \\ 0
	\end{pmatrix} \\
	& = e^{n g(z)} \left( R_{11}(z) N_{11}(z) + R_{12}(z) N_{21}(z) \right),
	\qquad z \in K.
	\end{align*}
	Using \cref{lem:asymptotics final transformation}, we calculate 
	\begin{align*}
	q_{n+k,n}(z) & = e^{ng(z)} \left( N_{11}(z) \left(1+\bigO\left(\frac{1}{n (1+|z|)}\right)\right) +  
		N_{12}(z) \bigO\left(\frac{1}{n(1+|z|)}\right)\right) \\
		& = N_{11}(z) e^{ng(z)}  \left(1+\bigO\left(\frac{1}{n(1+|z|)}\right) \right),
		\qquad z \in K.
	\end{align*}
	with $\bigO$ term that is uniform on $K$.
	The second identity holds since $|N_{11}(z)|$ is bounded away
	from $0$ for $z \in K$.	
	Then part (a) of \cref{prop:zerosqnn} follows since
	the $(1,1)$ entry of $N$ satisfies
	\[ N_{11}(z) = A_0(z) \psi(z)^k \]
	see \eqref{eq:standard solution model RHP} and \eqref{eq:presumed solution of model RHP}.

\subsection{Proof of \cref{prop:zerosqnn}(b)}

Part (b) follows from the strong asymptotic formula
\eqref{eq:qnkasymptotics} in a standard fashion. 
Indeed, \eqref{eq:qnkasymptotics} implies that
\[ \lim_{n \to \infty} \frac{1}{n} \log | q_{n,n}(z)| =
	\Re g(z), \qquad z \in \mathbb C \setminus \Sigma_0  \]
and $\Re g = - U^{\mu_0}$ where 
\[ U^{\mu_0}(z) = \int \log \frac{1}{|z-s|} d\mu_0(s) \]
is the logarithmic potential of $\mu_0$.
Part (b) then follows from the following result that is
well-known, but we state it here for convenience, as we will
also use it in later proofs.

\begin{lemma} \label{lem:weaklimit}
	Suppose $(p_n)_n$ is a sequence of monic polynomials
	with $\deg p_n = n$ whose zeros are all in a bounded subset 
	of the complex plane. Suppose $\mu$ is a probability measure
	with compact support such that
	\begin{equation} \label{eq:logpnlimit} 
		\lim_{n \to \infty} \frac{1}{n} \log |p_n(z)| = -U^{\mu}(z), \qquad
		\text{ a.e.\ on } \mathbb C. \end{equation}
	Then the sequence $(\nu(p_n))_n$ of normalized
	zero counting measures tends to $\mu$ in the weak sense.
\end{lemma}
\begin{proof}
Let $\nu_n = \nu(p_n)$ the normalized counting measure
of the zeros of $p_n$, see \eqref{eq:zero counting measure}.
The proof that $\mu_0$ is the weak limit of $(\nu_n)_n$ as $n \to \infty$ 
uses  basic tools from logarithmic potential theory. 

The probability measures $\nu_{n}$ are all supported 
on a fixed  compact subset $K$ of $\mathbb C$. The set of
probability measures on $K$ is compact for the weak
topology. Thus, by a standard compactness argument,
it suffices to show that $\mu$ is the only possible
limit of a weakly convergent subsequence.

Suppose $(\nu_{n_j})_j$ is a subsequence of $(\nu_n)_n$
that weakly converges to some probability measure 
$\nu$ on $K$.  Then by the  Lower Envelope Theorem, see \cite[Theorem  I.6.9]{saff_logarithmic_1997} one has
\[ \lim_{j \to \infty}
	U^{\nu_{n_j}}(z) = U^{\nu}(z), \quad
	\text{ for quasi-every } z \in \mathbb C. \]
Here quasi-every $z$ means except for a set of capacity zero,
and in particular it holds almost everywhere with respect
to two-dimensional Lebesgue measure.
Because of \eqref{eq:logpnlimit} 
and 
\[ U^{\nu_{n}}(z) = -\frac{1}{n} \log |p_{n}(z)|\]
we then
find that  $U^{\mu}(z) = U^{\nu}(z)$ a.e.\ on $\mathbb C$.
The uniqueness theorem for logarithmic potentials, see \cite[Theorem II.2.1]{saff_logarithmic_1997} then implies
that $\nu=\mu$. This proves the lemma.  
\end{proof}

The proof of \cref{prop:zerosqnn}(b) clearly also
works for the zeros of $q_{n+k,n}$ as $n \to \infty$ with a fixed $k \in \mathbb Z$.

\subsection{Proof of \cref{thm:zero distribution determinant}}

From \eqref{eq:identity_determinant} and the strong asymptotic formula \eqref{eq:qnkasymptotics} that
we use for $k=0$ and $k=1$, we obtain
\begin{align} \nonumber
	\det P_{n,n}(\zeta^2-\beta^2) & = 
	\frac{1}{2 \zeta} (q_{2n,2n}(\zeta) q_{2n+1,2n}(-\zeta) - q_{2n,2n}(-\zeta) q_{2n+1,2n}(\zeta)) \\
	& = \frac{1}{2\zeta} A_0(\zeta) A_0(-\zeta) 
	e^{2n (g(\zeta) + g(-\zeta))} \nonumber \\
	& \qquad \times  \nonumber
	\left(\psi(-\zeta) 	\left(1+\bigO\left(\frac{1}{n(1+|\zeta|)}\right) \right)
		- \psi(\zeta)	\left(1+\bigO\left(\frac{1}{n(1+|\zeta|)}\right) 
\right) \right)
\\  \label{eq:detPnnasymptotics}
	& = A_0(\zeta) A_0(-\zeta) \frac{\psi(-\zeta) - \psi(\zeta)}{2 \zeta} e^{2n (g(\zeta) + g(-\zeta))} 
	\left(1+\bigO\left(\frac{1}{n(1+|\zeta|)}\right) \right)
	\end{align}
as $n \to \infty$, uniformly for $\zeta$ in compact subsets of 
 $\overline{\C} \setminus (\Sigma_0 \cup -\Sigma_0)$.
 We may indeed combine the two $\bigO$ terms in the last step
 of \eqref{eq:detPnnasymptotics} since
 $\psi(-\zeta)-\psi(\zeta)$  does not  vanish in 
 $\C \setminus (\Sigma_0 \cup -\Sigma_0)$.
 
Thus the zeros of $\zeta \mapsto \det P_{n,n}(\zeta^2-\beta^2)$ tend to 
$\Sigma_0 \cup -\Sigma_0$ as $n \to \infty$, which means
that the zeros of $\det P_{n,n}$ tend to $\widetilde{\Sigma}_0$, see 
\eqref{eq:pushSigma}.

Using \eqref{eq:detPnnasymptotics} we also find
for $\zeta \in \mathbb C \setminus (\Sigma_0 \cup - \Sigma_0)$,
\begin{align*} 
	\lim_{n \to \infty}
		\frac{1}{2n} \log \left|\det P_{n,n}(\zeta^2 - \beta^2)\right|
	& = 
	\Re \left(g(\zeta) + g(-\zeta)\right) \\
	& = \int \log |\zeta-s| d\mu_0(s) + \int \log |\zeta+s| d\mu_0(s) \\
	& = \int \log |\zeta^2 - s^2| d\mu_0(s) 
\end{align*} 
and therefore for $z \in \mathbb C \setminus \widetilde{\Sigma}_0$,
\begin{align*} 
	\lim_{n \to \infty} \frac{1}{2n} \log \left|\det P_{n,n}(z)\right| 
	& = \int \log \left|z- (s^2 - \beta^2)\right| d\mu_0(s) 
	= - U^{\widetilde{\mu}_0}(z),
\end{align*} 
see also \eqref{eq:pushmu0}. Since  $\det P_{n,n}$ is a monic
polynomial of degree $2n$ we can apply \cref{lem:weaklimit} and 
 \cref{thm:zero distribution determinant} follows.

\section{Proof of \cref{thm:zero distribution top right entry}}\label{sec:proofofthm25}

Throughout this section we assume that the
inequality \eqref{eq:inequality g-function} for the
$g$-function holds.
We write
\begin{equation} \label{eq:defQn} 
	Q_n(z) = \frac{q_{n,n}(z) - q_{n,n}(-z)}{2z}, 
	\end{equation}
which is a polynomial of degree
\[ d_n = \deg Q_n \leq n-1. \]
We let $\nu(Q_n)$ be its normalized zero counting measure.

\begin{lemma} \label{lem:limit signed measure is a probability measure} 
\begin{enumerate}
	\item[\rm (a)] 
	The a priori signed measure $\nu$ defined by \eqref{eq:limitnu} 
	is a probability measure  with 
	\begin{align}  \label{eq:Unu} 
	U^{\nu}(z)  = \min \left( U^{\mu_0}(z), U^{\mu_0}(-z) \right)  
	= \begin{cases} U^{\mu_0}(z), & z \in \LHP, \\
	U^{\mu_0^*}(z), & z \in \RHP.
	\end{cases}  \end{align}
	
	\item[\rm (b)] We have
	\begin{equation} \label{eq:logQnlimit} 
		\lim_{n \to \infty} 
		-\frac{1}{n} \log |Q_n(z)| = U^{\nu}(z),
			\qquad z \in \mathbb C \setminus (i\mathbb R 
			\cup \Sigma_0 \cup -\Sigma_0),
			\end{equation}
		where the convergence is uniform in compacts of $\mathbb C \setminus (i\mathbb R 
		\cup \Sigma_0 \cup -\Sigma_0)$.
		
	\item [\rm (c)] The zeros of $Q_n$ tend to $i\mathbb R 
	\cup \Sigma_0 \cup -\Sigma_0$ or to $\infty$ as $n \to \infty$.
\end{enumerate}
\end{lemma}
\begin{proof}
(a)	The inequality \eqref{eq:inequality g-function} tells us that
	\begin{equation} \label{eq:nupositive1} 
		U^{\mu_0}(z)  > U^{\mu_0}(-z) = U^{\mu_0^*}(z),
	\qquad z \in \RHP \end{equation}
	where $\mu_0^*$ is the reflection of $\mu_0$
	in the imaginary axis. Because of symmetry we have
	the opposite inequality in the left half-plane and
	equality on the imaginary axis. This proves the
	second equality in \eqref{eq:Unu}.  
	
	Write $\mu_0 = \mu_R + \mu_L$ with $\mu_R = \mu_0 |_{\RHP}$, $\mu_L = \mu_0 |_{\LHP}$,
	and similarly $\mu_0^* = \mu^*_R + \mu^*_L$.  
	By the properties 	of balayage, we have
	$U^{\mu_L}(z) = U^{\Bal(\mu_L)}(z)$ and $U^{\mu^*_L}(z)  = U^{\Bal(\mu^*_L)}(z)$ for $z \in \RHP$,
	and then \eqref{eq:nupositive1} gives
	\begin{equation} \label{eq:nupositive2} 
	U^{\mu_R}(z) + U^{\Bal(\mu_L)}(z)  > 
	U^{\mu_R^*}(z) + U^{\Bal(\mu_L^*)}(z),
	\qquad z \in \RHP \end{equation}
	with equality for $\Re z = 0$.
	Both sides of \eqref{eq:nupositive2} are harmonic
	in the left half-plane, and behave like ${- \log|z| + \bigO(z^{-1})}$ as $|z| \to \infty$. Since they are equal on the imaginary
	axis, we obtain equality in the left half-plane
	by the maximum principle for harmonic functions, that is,
	\begin{equation} \label{eq:nupositive3} 	
		U^{\mu_R}(z) + U^{\Bal(\mu_L)}(z)
		 =  U^{\mu_R^*}(z) + U^{ \Bal(\mu^*_L)}(z),
		\quad z \in \overline{\LHP}. 
	\end{equation}	
	
	We next use De La Vall\'ee Poussin's theorem from potential theory which
	is stated in 	\cite[Theorem IV 4.5]{saff_logarithmic_1997}
	for measures with compact supports, and is extended to measures with unbounded support
	in  \cite[Theorem 4.9]{orive_equilibrium_2019}.
	We obtain from \eqref{eq:nupositive2} and
	\eqref{eq:nupositive3} that 
	\[ \left(\mu_R + \Bal(\mu_L) \right) |_{\overline{\LHP}}
	\leq \left(\mu^*_R + \Bal(\mu^*_L) \right) |_{\overline{\LHP}}. \]
	Since $\mu_R$ and $\mu^*_R$ are supported in the
	open right half-plane and the balayage measures
	are on the imaginary axis, we conclude 
	\[ \Bal(\mu_L) \leq \Bal (\mu^*_L)\]
	By symmetry we have $\Bal(\mu_R) = \Bal(\mu^*_L)$
	and we find $\nu_0 = \Bal(\mu_R) - \Bal(\mu_L) \geq 0$.
	Then also $\nu = \mu_L + \mu_R^* + \nu_0 \geq 0$
	is a positive measure.
	
	Furthermore
	\[ U^{\nu} = U^{\mu_L} + U^{\mu_R^*} 
		+ U^{\Bal(\mu_R)} - U^{\Bal(\mu_L)}. \]
	In the right half-plane we have $U^{\Bal(\mu_L)} = U^{\mu_L}$,
	and $U^{\Bal(\mu_R)}  = U^{\Bal(\mu_L^*)} = U^{\mu_L^*}$,
	and we find
	$ U^{\nu} = U^{\mu_R^*} + U^{\mu_L^*} =  U^{\mu_0^*}$ in the right half-plane.
	Similarly $U^{\nu} = U^{\mu_0}$ in the left half-plane
	and the first equality in \eqref{eq:Unu} follows.
	The equality implies that
	\[ U^{\nu}(z) = - \log |z| + \bigO(1) 
		\qquad \text{ as } z \to \infty, \]
	and therefore $\nu$ has total mass one, i.e., 
	it is a probability measure. 
	
\medskip
(b) The asymptotic formula \eqref{eq:qnkasymptotics} implies
that  
\begin{align*}
\lim_{n \to \infty} \frac{1}{n} \log|q_{n,n}(z)|
& = \Re g(z),  \qquad z \in \mathbb C \setminus \Sigma_0, \\
\lim_{n \to\infty} \frac{1}{n} \log |q_{n,n}(-z)| 
& = \Re g(-z),  \qquad z \in \mathbb C \setminus -\Sigma_0,
\end{align*}
where the convergence is uniform in compacts.
The inequality  \eqref{eq:inequality g-function} for the $g$-function then gives us for $z \in \mathbb C \setminus (i\mathbb R \cup \Sigma_0 \cup -\Sigma_0)$,
\begin{align*}
\lim_{n \to \infty} \frac{1}{n} \log|q_{n,n}(z)-q_{n,n}(-z)|
= \max(\Re g(z), \Re g(-z)), 	
\end{align*}
where the convergence is uniform in compacts of $\mathbb C \setminus (i\mathbb R 
\cup \Sigma_0 \cup -\Sigma_0)$.
By \eqref{eq:defQn}, this leads to 
\[ \lim_{n\to\infty} - \frac{1}{n} \log |Q_n(z)|
	= \min \left(U^{\mu_0}(z), U^{\mu_0}(-z) \right),
	\qquad z \in \mathbb C \setminus (i\mathbb R \cup \Sigma_0 \cup -\Sigma_0) \]
	uniformly on compacts.
Then \eqref{eq:logQnlimit} follows because of \eqref{eq:Unu}.


\medskip
(c) Because the convergence of \eqref{eq:logQnlimit} is uniform in compacts and $U^\nu$ is harmonic and hence finite in $\C \setminus (i \R \cup \Sigma_0 \cup -\Sigma_0)$, we find that $Q_n$ has no zeros in fixed compacts of $\C \setminus (i \R \cup \Sigma_0 \cup -\Sigma_0)$ for $n$ large enough.
Consequently, the zeros of $Q_n$ either tend to $i \R \cup \Sigma_0 \cup -\Sigma_0$ or run off to $\infty$.
\end{proof}

\begin{lemma}
\begin{enumerate}
\item[\rm (a)] 
	The first moment of $\mu_0$ is positive, that is, 
	 $m_1 = \int s d\mu_0(s) > 0$.
\item[\rm (b)] If $n$ is odd then $Q_n$ is
a monic polynomial of degree $d_n = n-1$. 
\item[\rm (c)] If $n$ is even and large enough, then $d_n = n-2$,
and the leading coefficient $\rho_n$ of $Q_n$ 
satisfies 
\[ \lim_{n \to \infty} \frac{\rho_n}{n} = - m_1 < 0. \]
\end{enumerate}
\end{lemma}
\begin{proof}
(a)	It is possible to explictly compute 
	$m_1 = \frac{5}{8} \alpha - \frac{3}{4} \sqrt{\alpha}
	+ \frac{5}{8}$ and conclude from there that $m_1 > 0$.
		We give another proof that only relies on the assumption
	\eqref{eq:inequality g-function} and thus generalizes
	 to other situations.
	
We note that for $|z|$ big enough, by a Taylor
expansion of the logarithm,
\begin{align*} 
	g(z) & = \int \log(z-s) d\mu_0(s) \\	
	& = \log z + \int \log\left(1-\frac{s}{z} \right) d \mu_0(s) \\
	& = \log z - \sum_{k=1}^{\infty} \frac{m_k}{k z^k}.
	\end{align*}
where $m_k = \int s^k d\mu_0(s)$ for $k =1,2, \ldots$.
Since $\log(-z) = \log z \pm \pi i$, with some
choice of $\pm$ depending on the location of $z$
and the precise branch of the logarithm, we find
\begin{equation} \label{eq:defhz} 
	g(-z) - g(z) - \pm \pi i  = 2 h(z) \qquad \text{ with }
	\quad 
	h(z) = \sum_{k=0}^{\infty} \frac{m_{2k+1}}{2k+1} z^{-2k-1}. 
	\end{equation}
Note that all moments are real by the symmetry of $\mu_0$
with respect to the real axis.
If $m_1$ would be negative, then we find by letting $z = x$ go 
to infinity along the positive real axis that
\[ \Re g(-x) - \Re g(x) = 2 h(x) 
	= 2m_1 x^{-1}  + \bigO(x^{-3}), \]
which is $< 0$ for $x$ big enough, which contradicts our assumption \eqref{eq:inequality g-function}. Thus $m_1 \geq 0$. 

Now assume $m_1 = 0$.
Then take the first $k\geq 1$ with $m_{2k+1} \neq 0$.
Such a $k$ has to exist, since otherwise $h$ is identically
zero, and thus $\Re g(-z) = \Re g(z)$ for every $|z|$
large enough, which contradicts \eqref{eq:inequality g-function}.
Then by \eqref{eq:defhz}, 
\[ \Re g(-z) - \Re g(z) 
	= 2 \frac{m_{2k+1}}{2k+1} \Re (z^{-2k-1}) + \bigO(z^{-2k-3}).
	\quad \text{ as } z \to \infty. \]
Taking $z = r e^{i\theta}$  we get
\[ \Re g(-z) - \Re g(z) 
	= 2 \frac{m_{2k+1}}{2k+1} r^{-2k-1} \cos \left((2k+1) \theta\right)
		+ \bigO(r^{-2k-3}) \quad
		\text{ as } r \to \infty. \]
Because of \eqref{eq:inequality g-function}, we find
$m_{2k+1} \cos\left((2k+1) \theta\right) \geq 0$  for every 
$-\frac{\pi}{2} < \theta < \frac{\pi}{2}$
which is clearly impossible since $m_{2k+1} \neq 0$
and the cosine changes sign on the interval $(-\frac{\pi}{2},
\frac{\pi}{2})$ as $2k+1 \geq 3$.

\medskip
(b)  This is clear from the formula \eqref{eq:defQn}
since $q_{n,n}$ is a monic polynomial of degree $n$.

\medskip
(c) 	
If $n$ is even then the leading terms  of  $q_{n,n}(z)$ and $q_{n,n}(-z)$ cancel when we take their difference as we do in \eqref{eq:defQn} 
and we find that  $Q_n$ is  a polynomial of degree 
$\leq n-2$ whose coefficient for $z^{n-2}$ 
is the second coefficient of $q_{n,n}$, i.e., the coefficient $\rho_n$
in 
\[ q_{n,n}(z) = z^n + \rho_n z^{n-1} + O(z^{n-2}) \text{ as } z \to \infty. \]
Then $-\rho_n$ is equal to the sum of the 
zeros of $q_{n,n}$, which is 
\[ - \rho_n	= n \int s d\nu_n(s)  \]
where $\nu_n = \nu(q_{n,n})$ is the normalized zero counting measure of $q_{n,n}$. 
Since $\nu_n \to \mu_0$ by \cref{prop:zerosqnn} we obtain
that  $-\rho_n/n \to \int s d\mu_0(s) = m_1$ as $n \to \infty$.
Then part (c) of the lemma follows, since $m_1 > 0$ by part (a).
\end{proof}

\begin{proof}[Proof of \cref{thm:zero distribution top right entry}.]
Take an arbitrary  $a \in \mathbb C \setminus (i \mathbb R \cup \Sigma_0 \cup -\Sigma_0)$ that is distinct from any of the zeros of $Q_n$ for $n=1,2, \ldots$. 
By \eqref{eq:logQnlimit} we have that 
\begin{equation} \label{eq:logQna} 
	\lim_{n\to\infty} -\frac{1}{n} \log |Q_n(a)| = U^{\nu}(a).
\end{equation}
Let $T_a$ be the M\"obius transformation
$z \mapsto  T_a(z) = \frac{1}{z-a}$. Then 
\begin{equation} \label{eq:Qntransform} 
		z \mapsto \frac{z^n}{Q_n(a)} Q_n\left(z^{-1} + a\right) 
		\end{equation}
is a monic polynomial of degree $n$ whose zeros are
$T_a(z_{j,n})$, $j=1, \ldots, d_n$ where $z_{j,n}$, $j=1,
\ldots, d_n$ are the zeros of $Q_n$, 
together with a zero of order $n-d_n$ at $z=0$. 
It follows from \cref{lem:limit signed measure is a probability measure}(c) 
that there exists a $\delta > 0$ such that $|z_{j,n} - a| > \delta$ for $j = 1, \ldots, d_n$,
so the set of zeros of the polynomials \eqref{eq:Qntransform} is bounded.

We have by \eqref{eq:logQnlimit} and \eqref{eq:logQna},
whenever $z^{-1} + a \not\in i\mathbb R \cup \Sigma_0 \cup - \Sigma_0$,
\begin{align} \nonumber
	\lim_{n \to \infty}
	- \frac{1}{n} \log \left| \frac{z^n}{Q_n(a)} Q_n\left(z^{-1} + a\right) \right| 
	& = U^{\nu}\left(z^{-1} + a\right) - U^{\nu}(a) - \log |z| \\
	\nonumber
	& = \int \left( \log \frac{1}{|z^{-1}+a-s|} 
		- \log \frac{1}{|a-s|} - \log |z| \right) d\nu(s) \\
		\nonumber
	& = \int \log \frac{1}{|z- T_a(s)|} d\nu(s) \\ \label{eq:logQntransform} 
	& = U^{T_a^* \nu}(z) \end{align}
where $T_a^* \nu$ is the pushforward of $\nu$ under the mapping $T_a$.
From \cref{lem:weaklimit} we then find that $T_a^* \nu$ is 
the weak limit of the normalized zero counting measures of the
polynomials \eqref{eq:Qntransform}, and this means that $\nu$
is the weak limit of the zeros of $Q_n$, as claimed in \cref{thm:zero distribution top right entry}.

Since $(P_n)_{12}(z) = Q_{2n,2n}(\zeta)$ with $\zeta = (z + \beta^2)^{1/2}$ by \eqref{eq:Pnn12} and \eqref{eq:defQn}
we then also find that $\widetilde{\nu}$ is the
limit of the zeros of $(P_n)_{12}$ and the theorem is
proved.
\end{proof}

\section{Proof of \cref{thm:geometrical condition}} \label{sec:proofofthm26}

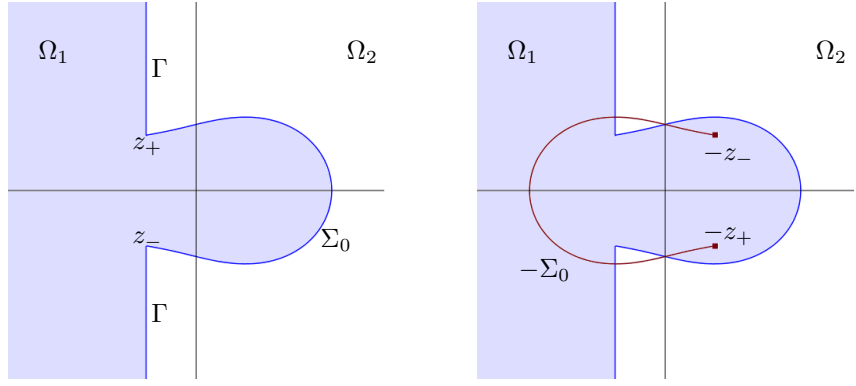
\begin{figure} \vspace*{-5mm}
	\centering
\begin{overpic}[height=5cm,keepaspectratio]{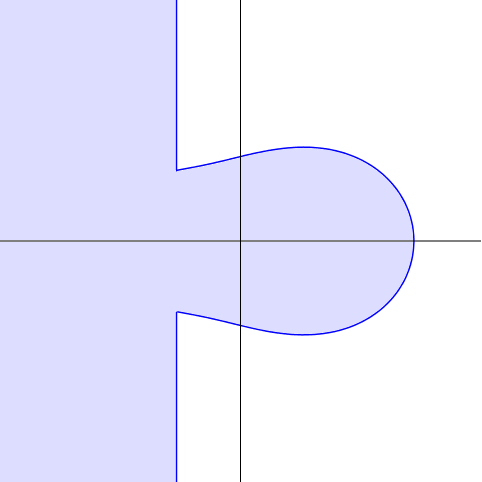}
		\put(8,85){$\Omega_1$}
		\put(90,85){$\Omega_2$}
		\put(38,15){$\Gamma$}
		\put(38,80){$\Gamma$}
		\put(83,35){$\Sigma_0$}
		\put(33,61){$z_+$}
		\put(33,36){$z_-$}
	\end{overpic}
	\hspace{1cm}
\begin{overpic}[height=5cm,keepaspectratio]{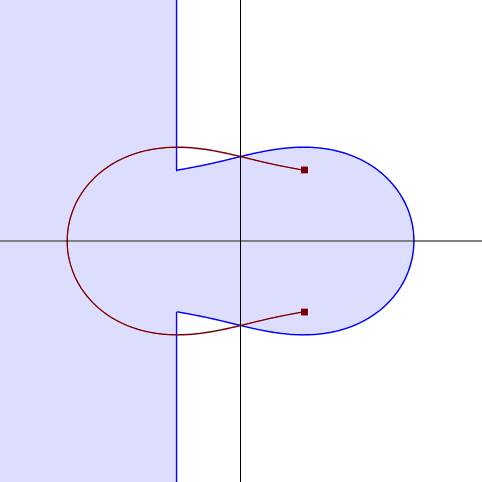}
		\put(8,85){$\Omega_1$}
		\put(90,85){$\Omega_2$}
		\put(11,27){$-\Sigma_0$}
		\put(60,58){$-\zminus$}
		\put(60,37){$-\zplus$}
	\end{overpic}
	\caption{\label{fig:Omega12} The sets $\Omega_1$ (in blue) and $\Omega_2$ as used in the proof of \cref{thm:geometrical condition}. In the right picture, the set $-\Sigma_0$ is plotted in red and we see that indeed $-\Sigma_0 \cap \RHP$ is contained in $\Omega_1$.}
\end{figure}

\subsection{Proof of \cref{thm:geometrical condition}}
For the proof of  \cref{thm:geometrical condition} 
we will work with 
\begin{equation} \label{eq:D0z} 
D_0(z) = ((z-z_+)(z-z_-))^{1/2}, \qquad z \in \mathbb C \setminus \Gamma, \end{equation}
with branch cut along 
\begin{equation} \label{eq:Gamma}
\Gamma = \{ z \in \mathbb C \mid \Re z = \Re z_+, \	|\Im z| \geq \Im z_+ \}  
\end{equation}
and such that $D_0(z)$ is real and positive for $z \in \mathbb R$.
Then $\Gamma \cup \Sigma_{0}$ is a contour that separates
the complex plane into two unbounded domains $\Omega_1$ and $\Omega_2$,
i.e., $\Omega_1 \cup \Omega_2 = \mathbb C \setminus (\Gamma\cup \Sigma_{0})$, where $\Omega_1$ lies to the left and
$\Omega_2$ lies to the right of $\Gamma \cup \Sigma_{0}$,
see left panel of \cref{fig:Omega12}.  
The assumption of \cref{thm:geometrical condition} says
that $-\Sigma_{0} \cap \RHP \subset \Omega_1$ as
shown in the right panel of \cref{fig:Omega12}.

Here we use the fact that $\Sigma_0$  intersects $\Gamma$
only in $z_{\pm}$ which is obvious from the figures.
It will follow from \cref{lem:lemma71} that we state
and prove below after finishing the proof of 
\cref{thm:geometrical condition}.
In fact from part (c) of that lemma we conclude that 
\begin{equation} \label{eq:Rez on Sigma0}
\Sigma_0 \setminus \{z_-,z_+\} \subset \{ z \in \mathbb C \mid \Re z > \Re z_+ \}. 
\end{equation}

Because of \eqref{eq:gprimeidentity} and \eqref{rationalFunctionQ}
we have
\begin{equation} \label{eq:gprimeidentity2}
2 g'(z) = V_{\alpha}'(z)
+ 2 \frac{z+ \frac{(1-\sqrt{\alpha})^2}{2}}{(z-\beta)(z+\beta)(z+\alpha+\beta)}  D(z) \end{equation}
where $D(z)$ is the branch of the square root \eqref{eq:D0z}
that is analytic in $\mathbb C \setminus \Sigma_0$, namely
\begin{align} \label{eq:D0andD}
D(z) = \begin{cases} -D_0(z), & \text{ for } z \in \Omega_1,\\
D_0(z), & \text{ for } z \in \Omega_2.
\end{cases} 
\end{align}

We apply  a partial fraction decomposition to
the rational expression in front of $D(z)$ in \eqref{eq:gprimeidentity2} 
and we combine it with $V_{\alpha}'(z)$ that
we get from \eqref{eq:Valpha} to obtain for $z \in \Omega_1$,
\begin{multline*}
2 g'(z) =  - \frac{1}{D_0(\beta)} \left[ \frac{D_0(z) - D_0(\beta)}{z-\beta} \right]
- \frac{1}{D_0(-\beta)}   \left[ \frac{D_0(z) - D_0(-\beta)}{z+\beta} \right]
\\
+ \frac{2}{D_0(-\alpha-\beta)} \left[ \frac{D_0(z) - D_0(-\alpha-\beta)}{z+\alpha + \beta} \right],
\qquad z \in \Omega_1.
\end{multline*}
Since $g'(z) \to 0$ as $z \to \infty$, we have
$D_0(\beta)^{-1}
+ D_0(-\beta)^{-1} = 2 D_0(-\alpha-\beta)^{-1} $, 
and thus 
\begin{equation} \label{eq:gprime2}
g'(z) = h(z), \qquad z \in \Omega_1,
\end{equation}
with 
\begin{multline} \label{eq:hz}  
h(z) =   -\frac{1}{2 D_0(\beta)} \left[ \frac{D_0(z) - D_0(\beta)}{z-\beta}  -  \frac{D_0(z) - D_0(-\alpha-\beta)}{z+\alpha + \beta} \right] \\
- \frac{1}{2 D_0(-\beta)}   \left[ \frac{D_0(z) - D_0(-\beta)}{z+\beta} -  \frac{D_0(z) - D_0(-\alpha-\beta)}{z+\alpha + \beta} \right],
	\qquad z \in \mathbb C \setminus \Gamma.
\end{multline}
We emphasize that we consider $h(z)$ for $z \in \mathbb C \setminus \Gamma$,
but the identity \eqref{eq:gprime2} only holds for
$z \in \Omega_1$.

The expressions within square brackets in \eqref{eq:hz}
turn out to have positive real parts. This is
a consequence of Lemma \ref{lem:lemma72} that
we state and prove separately below.
Note that $D_0(x) > 0$ for $x \in \mathbb R$,
because of the choice of the square root in \eqref{eq:D0z}. 
Then we conclude from \eqref{eq:hz} that
\begin{equation} \label{eq:Rehz isnegative} 
\Re h(z) < 0, \qquad z \in \mathbb C \setminus \Gamma. \end{equation}

Now take $z_0 \in -\Sigma_0 \cap \RHP$, $z_0 \neq - z_{\pm}$. 
By the geometrical
assumption in \cref{thm:geometrical condition}, we then have that
$z_0 \in \Omega_1$. 
Then $-\overline{z}_0$ belongs to $\Sigma_0 \setminus \{z_\pm\}$, and because of \eqref{eq:Rez on Sigma0} we then have 
\begin{equation} \label{eq:z0 segment} 
[-\overline{z}_0, z_0] \subset \mathbb C \setminus \Gamma. 
\end{equation}
where $[-\overline{z}_0,z_0]$ denotes the horizontal 
line segment from $-\overline{z}_0$ to $z_0$.

If this segment would be fully contained in $\Omega_1$,
then we can conclude from \eqref{eq:gprime2},  
\eqref{eq:z0 segment} and the fundamental theorem of calculus
that 
\begin{equation} \label{eq:gdifference} 
	g(z_0) - g_+(-\overline{z}_0)
	= \int_{[-\overline{z}_0,z_0]} h(s) ds  
	\end{equation}
where we use $g_+(-\overline{z}_0)$ to indicate that we use the $+$ boundary value of $g$ at $-\overline{z}_0 \in \Sigma_0$.
If $[-\overline{z}_0,z_0]$ is not fully contained in $\Omega_1$,
then we can write
\[ g(z_0) - g_+(-\overline{z}_0)
= \int_{\gamma_{-\overline{z}_0,z_0}} h(s) ds  
\] 
where $\gamma_{\overline{z}_0,z_0}$ is any path from
$-\overline{z}_0$ to $z_0$ in $\Omega_1$. 
By Cauchy's theorem and \eqref{eq:z0 segment},
we can then deform to the horizontal line segment, 
since $h$ is analytic in $\mathbb C \setminus \Gamma$.
Thus \eqref{eq:gdifference} holds in all cases.

Taking real parts in \eqref{eq:gdifference},
and noting  \eqref{eq:Rehz isnegative}, we find that
that
\begin{equation} \label{eq:Reg on minSigma0} 
\Re (g(z_0) - g(-\overline{z}_0)) < 0, \qquad z_0 \in -\Sigma_0
	\cap \RHP  \setminus  \{-z_{\pm}\}. \end{equation}
The above argument can be easily adapted to the case $z_0 = - z_{\pm}$ and we also find
\begin{equation} \label{eq:Reg on minzpm} 
\Re (g(z_0) - g(-\overline{z}_0)) < 0, \qquad z_0 = - z_{\pm}. \end{equation}
	
We finally extend the inequality \eqref{eq:Reg on minSigma0},
\eqref{eq:Reg on minzpm}
to the full right-half plane by a subharmonicity argument.
We use that $\Re g(z)$ is harmonic in $\mathbb C \setminus \Sigma_0$ and subharmonic on $\mathbb C$, and 
$\Re g(-z)$ is harmonic in $\mathbb C \setminus -\Sigma_0$.
Also $\Re g(z) = \Re g(\overline{z})$ by symmetry in
the real axis, and it follows that
\[ \Re (g(z) - g(-z)) = \Re(g(z) - g(-\overline{z})) \]
is subharmonic on $\RHP \setminus -\Sigma_0$. It has boundary value $0$ on the imaginary axis, and at infinity.
In addition it is  $< 0$ on $-\Sigma_0 \cap \RHP$ by
\eqref{eq:Reg on minSigma0} and \eqref{eq:Reg on minzpm}.  
 Then the maximum principle for subharmonic functions 
 tells us that 
$\Re (g(z) - g(-z)) < 0$ for $\Re z > 0$,
which concludes the proof of \cref{thm:geometrical condition}
pending the proofs of two lemmas that we will 
turn to next.

\subsection{\cref{lem:lemma71}}

Lemma 6.4 in \cite{charlier2019periodic} provides information
on the critical trajectories of $\widehat{Q}_{\alpha} dz^2$,
that is directly translated into information on $\Sigma_0$
because of the relation \eqref{eq:QvsQtilde}. This is
contained in part (a) of the following lemma.

We assume that $\Sigma_0$ is oriented from $z_-$ to $z_+$.
\begin{lemma} \label{lem:lemma71}
	If we follow $\Sigma_0$ according to its orientation, then
	we have the following.
\begin{enumerate}
\item[\rm (a)] $z \mapsto |z+\alpha+\beta|$ is
	strictly increasing along $\Sigma_0 \cap \mathbb C^-$ 
		and by symmetry strictly decreasing along $\Sigma_0 \cap \mathbb C^+$.
	\item[\rm (b)] $Q_{\alpha}$ is real and positive on the real line,
	and on a smooth contour $L$ contained in the disk
	$|z+\alpha+\beta| \leq \sqrt{\alpha}$, and nowhere else in
	the complex plane.
	\item[\rm (c)]	
	The real part of $z$ is strictly increasing along
	$\Sigma_{0} \cap \mathbb C^-$ and strictly
	decreasing along $\Sigma_{0} \cap \mathbb C^+$.
	\end{enumerate}
\end{lemma} 
\begin{proof}
(a) This follows from \cite[Lemma 6.4 (a)]{charlier2019periodic}
	and the mapping $z \mapsto -z - \alpha - \beta$ that
	maps the trajectories of $\widehat{Q}_{\alpha} dz^2$ (that are
	relevant in \cite{charlier2019periodic}) to the trajectories
	of $Q_{\alpha}dz^2$.

\medskip
(b)
From \eqref{rationalFunctionQ} it is clear that $Q_{\alpha}$ is
real and positive on the real line. 
Furthermore, each value $w$ is taken by $Q_{\alpha}$ six times in $\mathbb C$
(counting multiplicities), since $w = Q_{\alpha}(z)$ can 
be written as a degree six polynomial equation for $z$, see
again \eqref{rationalFunctionQ}.
An inspection of the graph of $Q_{\alpha}$ on the real
line shows that $Q_{\alpha}$ has a 	strictly positive local minimum
at a point $x_\text{min} \in (-\alpha-\beta, - \beta)$,
and each value in $[Q_{\alpha}(x_\text{min}), \infty]$
is taken on the real line six times, and so these
values do not appear anywhere else in the complex
plane as function values of $Q_{\alpha}$. Moreover, each value in $[0,Q_{\alpha}(x_\text{min}))$
is taken four times on the real line (counting multiplicities), 
and therefore each of these values is taken exactly two times 
away from the real axis.

From $x_\text{min}$ there is then a smooth contour $L$ into the complex plane, orthogonal to the real line at $x_\text{min}$ along which  $Q_{\alpha}$ is real
and positive, 	and strictly decreasing if we move away from $x_\text{min}$.
The contour $L$ will end at $z_{\pm}$ since these are the two simple zeros
of $Q_{\alpha}$ away from the real line, see \eqref{rationalFunctionQ}.

Let $C$ be the circle $|z+\alpha+\beta| = \sqrt{\alpha}$ of radius
$\sqrt{\alpha}$ around $-\alpha-\beta$.
From $-\alpha - \beta < x_\text{min} < - \beta$ and $\alpha \leq 1$, 
it follows that
$x_\text{min}$ lies inside the circle, and the proof of (b) will be
finished if we can show that $L$ intersects the circle only in 
its endpoints $z_{\pm}$.

Suppose, to get a contradiction, that $L$ intersects $C$ at
a point $\hat{z} \neq z_{\pm}$. Then  $Q_{\alpha}(\hat{z}) > 0$
which implies that the trajectory of the quadratric differential
$Q_{\alpha} dz^2$ that passes through $\hat{z}$ has a 
vertical tangent at $\hat{z}$.  

If $\hat{z} \in \Sigma_2$ (see \cref{fig:zeros40GK}),
then $\Sigma_2$ is this trajectory passing through $\hat{z}$,
and we get a contradiction since $\Sigma_2$ is part of the circle $C$,
and $\Sigma_2$ does not have a vertical tangent at a point $\hat{z} \not\in \mathbb R$.

Thus $\hat{z} \in C \setminus \Sigma_2$. But $C \setminus \Sigma_2$
consists of vertical trajectories of the quadratic differential $Q_{\alpha} dz^2$ from $z_{\pm}$ to the double zero $- \frac{(1-\sqrt{\alpha})^2}{2}$,
which also lies on $C$. Since $Q_{\alpha}(\hat{z}) > 0$ we have
that the vertical trajectory has a horizontal tangent at $\hat{z}$.
A vertical tangent to the circle $|z+\alpha+\beta| = \sqrt{\alpha}$
only happens at the top and bottom points where the real part is $-\alpha-\beta$
but this is not the case at $\hat{z}$, since 
\[ \Re \hat{z} > \Re z_+ = - \frac{(1+\sqrt{\alpha})^2}{8} \]
see \eqref{simpleZerosRationalFunctionQ}, which is $> - \alpha -\beta$.

The contradiction shows that $L$ meets the circle $C$ only in $z_{\pm}$
and part (b) is proved.

\medskip
	
(c) By part (a) we have that $\Sigma_{0}$ is outside of the
closed disk $|z+\alpha+\beta| \leq \sqrt{\alpha}$,
except for the endpoints $z_{\pm}$ which are on the circle $C$.

Then it follows from part (b) that $Q_{\alpha}(z) \in \mathbb C \setminus
\mathbb R^+$ for every $z \in \Sigma_0 \setminus \{ z_{\pm}, x^*\}$
where $x^*$ is the point of intersection of $\Sigma_0$ with the positive
real axis. Then the trajectory $\Sigma_0$ does not have a vertical
tangent at any $z \in  \Sigma_0 \setminus \{z_{\pm}, x^*\}$,
which means that $\Re z $ is strictly increasing along $\Sigma_0 \cap \mathbb C^-$ according to its orientation from $z_-$ to $z_+$, 
and by symmetry strictly decreasing along $\Sigma_0 \cap \mathbb C^+$.
This proves part (c).
\end{proof}

\subsection{\cref{lem:lemma72}}

\begin{lemma} \label{lem:lemma72} For every $x \in \mathbb R$ and
	$z \in \mathbb C \setminus \Gamma$, we have that
	\begin{equation} \label{eq:ReDoineq} \frac{\partial}{\partial x} \Re \left(  \frac{D_0(z) - D_0(x)}{z-x} \right)> 0,
	\end{equation}
\end{lemma}
\begin{proof}
	Let $\Delta(w) = (w^2+1)^{1/2}$ with branch cut
	along $(-i \infty,-i] \cup [i, i\infty)$. Then it is easy
	to see from the definition \eqref{eq:D0z} that
	\begin{equation} \label{eq:D0andDelta} 
	\frac{D_0(z) - D_0(x)}{z-x} = \frac{\Delta(w) - \Delta(\xi)}{w-\xi} \end{equation}
	where $w = (z-\Re z_+)/\Im z_+$ and $\xi = (x-\Re z_+)/\Im z_+$.
	
	On its branch cut, $\Delta(w)$ has purely imaginary boundary
	values, namely $\Delta_{\pm}(i\eta) = \pm i \sqrt{\eta^2-1}$
	for $w = i \eta \in \Gamma$ with $\eta > 1$. Then  
	\begin{align*} 
	\frac{\Delta_{\pm} (w) - \Delta(\xi)}{w-\xi}
	& = \frac{\sqrt{\xi^2+1} \mp i \sqrt{\eta^2-1}}{\xi - i \eta} 
	\\
	& = \frac{\left(\sqrt{\xi^2+1} \mp i \sqrt{\eta^2-1}\right)(\xi + i \eta)}{\xi^2 + \eta^2} 
	\end{align*}
	and 
	\begin{align*}  
	\Re \left(
	\frac{\Delta_{\pm} (w) - \Delta(\xi)}{w-\xi}  \right)
	= \frac{\xi \sqrt{\xi^2+1} \pm \eta \sqrt{\eta^2-1}}{\xi^2 + \eta^2}.
	\end{align*}
	Then a little calculation shows that 
	\begin{align*}
	\frac{\partial}{\partial \xi} \Re \left(
	\frac{\Delta_{\pm} (w) - \Delta(\xi)}{w-\xi}  \right)
	= 
	\frac{\left(\eta \sqrt{\xi^2+1} \pm \xi \sqrt{\eta^2-1}\right)^2}{
		\sqrt{\xi^2+1} (\xi^2 + \eta^2)^2}.
	\end{align*}
	Since the numerator is a perfect square we find
	\begin{align} \label{eq:Deltaderiv} 
	\frac{\partial}{\partial \xi} \Re \left(
	\frac{\Delta_{\pm} (w) - \Delta(\xi)}{w-\xi}  \right) > 0
	\end{align}
	for $w = i \eta$, $\eta > 1$ and $\xi \in \mathbb R$.
	By symmetry the same inequality inequality \eqref{eq:Deltaderiv} holds  for $w = -i \eta$, $\eta < 1$
	and $\xi \in \mathbb R$.
	
	It follows that 
	\[ w \mapsto 
	\frac{\partial}{\partial \xi} \Re \left(
	\frac{\Delta_{\pm} (w) - \Delta(\xi)}{w-\xi}  \right) \]
	has positive boundary values on $(-i\infty-i] \cup [i, i\infty)$ and it is harmonic everywhere else in
	the complex plane. By the minimum principle for
	harmonic function, we get that the inequality
	\begin{align} \label{eq:Deltaderiv2} 
	\frac{\partial}{\partial \xi} \Re \left(
	\frac{\Delta(w) - \Delta(\xi)}{w-\xi}  \right) > 0
	\end{align}
	holds for every $w \in \mathbb C$.
	Because of the identity \eqref{eq:D0andD} and
	\eqref{eq:Deltaderiv2} we have
	\begin{align*}
	\frac{\partial}{\partial x} \Re \left(
	\frac{D_0(z) - D_0(x))}{z-x}  \right)  =
	\frac{1}{\Im z_+} 
	\frac{\partial}{\partial \xi} \Re \left(
	\frac{\Delta(w) - \Delta(\xi)}{w-\xi}  \right) > 0
	\end{align*}
	where $w = (z-\Re z_+)/\Im z_+$ and $\xi = (x-\Re z_+)/\Im z_+$,
	and the lemma follows.
\end{proof}

\subsection*{Acknowledgement}
We thank Christophe Charlier for useful comments and for 
allowing us to use the Figure~\ref{fig:low alpha tiling}.

\bibliographystyle{plain}
\bibliography{bibliography}

\end{document}